\documentclass[onefignum,onetabnum]{siamart190516}

\usepackage{lipsum}
\usepackage{amsfonts}
\usepackage{graphicx}
\usepackage{graphics}
\usepackage{epstopdf}
\usepackage{amsmath,amsfonts,booktabs,multirow}
\usepackage{algorithmic}
\usepackage{caption}
\ifpdf
  \DeclareGraphicsExtensions{.eps,.pdf,.png,.jpg}
\else
  \DeclareGraphicsExtensions{.eps}
\fi

\DeclareMathOperator*{\Argmin}{Argmin}
\DeclareMathOperator*{\Argmax}{Argmax}
\newcommand{\N}{\mathbb{N}}
\newcommand{\Z}{\mathbb{Z}}
\newcommand{\R}{\mathbb{R}}
\newcommand{\Q}{\mathbb{Q}}

\newsiamremark{assumption}{Assumption}
\newsiamremark{remark}{Remark}
\newsiamremark{hypothesis}{Hypothesis}
\crefname{hypothesis}{Hypothesis}{Hypotheses}
\newsiamthm{claim}{Claim}

\headers{Decomposition Methods for Global Solutions of MILPs}{K. Sun, M. Sun, and W. Yin}

\title{Decomposition Methods for Global Solution of Mixed-Integer Linear Programs 
\thanks{Submitted March 29, 2022; revised December 21, 2022 and September 1, 2023; accepted December 28, 2023.
}}

\author{Kaizhao Sun
\thanks{School of ISyE, Georgia Institute of Technology (\email{ksun46@gatech.edu}).}
\and 
Mou Sun
\thanks{Decision Intelligence Lab, DAMO Academy, Alibaba Group
  (\email{mou.sunm@alibaba-inc.com}).}
\and 
Wotao Yin
\thanks{Decision Intelligence Lab, DAMO Academy, Alibaba Group (US)
  (\email{wotao.yin@alibaba-inc.com}).}
}

\usepackage{amsopn}

\ifpdf
\hypersetup{
  pdftitle={Decomposition Methods for Global Solutions of MILPs},
  pdfauthor={K. Sun, M. Sun, and W. Yin}
}
\fi

\begin{document}

\maketitle

\begin{abstract}
This paper introduces two decomposition-based methods for two-block mixed-integer linear programs (MILPs), which aim to take advantage of separable structures of the original problem by solving a sequence of lower-dimensional MILPs. The first method is based on the $\ell_1$-augmented Lagrangian method (ALM), and the second one is based on a modified alternating direction method of multipliers (ADMM). In the presence of certain block-angular structures, both methods create parallel subproblems in one block of variables, and add nonconvex cuts to update the other block; they converge to globally optimal solutions of the original MILP under proper conditions. 
Numerical experiments on  three classes of MILPs demonstrate the advantages of the proposed methods on structured problems over the state-of-the-art MILP solvers.
\end{abstract}

\begin{keywords}
  Mixed-integer linear programs, decomposition method, augmented Lagrangian method, alternating direction method of multipliers
\end{keywords}

\begin{AMS}
  49M27, 
  90C11, 
  90C26  
\end{AMS}

\section{Introduction}
We consider the generic two-block mixed-integer linear program (MILP): 
\begin{align}\label{eq: MIP}
	p^* := \min_{x,z}\{ c^\top x + g^\top z:~ Ax + Bz = 0, \ x\in X, \ z \in Z\}, 
\end{align}
with decision variables $x\in \R^n$ and $z\in\R^d$, rational parameters $c \in \Q^n$, $g\in \Q^d$, $A \in \Q^{m \times n}$, and $B \in \Q^{m \times d}$, and compact sets $X\subset \R^n$ and $Z\subset \R^d$. Interesting to us is when certain entries of $x$ and $z$ must be integers. The use of the zero vector in $Ax+Bz=0$ has no loss of generality\footnote{Any $Ax+Bz=b$ can be equivalently expressed as $Ax+\tilde{B}\tilde{z} = 0$ where $\tilde{B} = [B, b]$ and $\tilde{z} = [z^\top, -1]^\top$.}. 

Despite the fast development of general-purpose MILP solvers over the past few decades \cite{achterberg2013mixed, bixby2012brief}, large MILPs still pose nontrivial challenges. This paper develops parallelizable methods for MILP \eqref{eq: MIP} with a decomposable structure.

\subsection{Decomposable Structure}
Consider the classic block-angular structure:
\begin{equation}\label{eq: block}
[A|B] = 
\left[ \begin{array}{ccc|c} 
    A_1 &        &     & B_1 \\ 
        &   \ddots&     & \vdots\\
        &          & A_P & B_P
\end{array}\right],\quad  X=X_1\times \dots \times X_P,
\end{equation}
where $A$ is a block-diagonal matrix, rows of $B$ are divided into groups accordingly, and the constraint $x\in X$ reduces to $x_1\in X_1,~\dots,~x_P\in X_P$. This structure naturally arises in two-stage optimization problems \cite{benders1962partitioning}, where the first-stage variable $z$ couples with the second-stage variables $x_i$ via $A_ix_i + B_iz = 0$ in each scenario $i \in [P]:=\{1,\cdots, P\}$. In decentralized optimization \cite{shi2014linear}, the consensus between neighboring agents $i$ and $j$ in a graph is written as $x_i = z_{ij}$ and $x_j = z_{ij}$, which is a special case of \eqref{eq: block}.

\subsection{ALM and ADMM}
When $X$ and $Z$ are convex, we can solve \eqref{eq: MIP} by the augmented Lagrangian method (ALM)~\cite{hestenes1969multiplier,powell1969method} and the alternating direction method of multipliers (ADMM)~\cite{gabay1976dual,glowinski1975approximation}. Define the augmented Lagrangian (AL) function: 
\begin{equation}\label{eq: AL}
	\mathcal{L}(x,z,\lambda, \rho) := c^\top x+ g^\top z + \langle \lambda, Ax+Bz\rangle + \rho \sigma(Ax+Bz), \ \lambda \in \R^m, \ \rho > 0.
\end{equation}
Classic ALM and ADMM use the penalty function 
$\sigma(\cdot) = \frac{1}{2}\|\cdot\|^2_2$.
The standard ALM is the iteration:
\begin{subequations}
    \begin{align}
       (x^{k+1},z^{k+1}) \in  & \Argmin_{x \in X,z\in Z} \mathcal{L}(x, z,\lambda^k, \rho^k),\\ \label{eq: alm-subproblem}
       \lambda^{k+1} = & \lambda^k + \rho^k (Ax^{k+1} +Bz^{k+1}),~~\text{update }\rho^{k+1}\text{ if needed.}
    \end{align}
\end{subequations}
This method cannot directly solve \eqref{eq: MIP} or take advantage of structure \eqref{eq: block}. {To make the former possible, a function $\sigma(\cdot)$ supporting the exact--penalty property is required. This paper will introduce ways to achieve the latter.}

The standard ADMM is the iteration: 
\begin{subequations}\label{eq: raw admm}
	\begin{align}
		x^{k+1} \in  & \Argmin_{x \in X} \mathcal{L}(x, z^k,\lambda^k, \rho),\label{eq: raw admm x} \\
		z^{k+1} \in  & \Argmin_{z \in Z} \mathcal{L}(x^{k+1}, z,\lambda^k, \rho), \label{eq: raw admm z}\\
		\lambda^{k+1} = & \lambda^k + \rho (Ax^{k+1} +Bz^{k+1}). \label{eq: raw admm dual}
	\end{align}
\end{subequations}
If structure \eqref{eq: block} is present, then the step \eqref{eq: raw admm x} decomposes into $P$ independent lower-dimensional subproblems. {However, \eqref{eq: raw admm} can diverge or only converge to an infeasible solution for \eqref{eq: MIP}.} 

\subsection{Our Approach}
This paper introduces two methods for \eqref{eq: MIP} that can yield parallel subproblems by utilizing \eqref{eq: block}, i.e., they solve a set of independent, parallel MILP subproblems, each involving one $x_i \in X_i$. The methods also offer a decomposition between $x$ and $z$ when $P=1$. Both methods converge to global solutions of \eqref{eq: MIP} under proper conditions. Below we present them at a high level.

We use the \underline{$\ell_1$ penalty} $\sigma(\cdot) = \|\cdot\|_1$ in the AL function \eqref{eq: AL} {due to its exact penalization property and component-wise separability.} Define the \textit{AL dual function} and the \textit{AL dual problem} for \eqref{eq: MIP} as: 
\begin{align}
	& d(\lambda, \rho) := \min_{x\in X, z\in Z}\mathcal{L}(x,z,\lambda, \rho), \label{eq: dual function}\\
	& \sup_{ \lambda \in \R^m, \rho >0} d(\lambda, \rho).\label{eq: dual problem}
\end{align}
Note that penalty parameter $\rho$ is not fixed in \eqref{eq: dual problem}.  
Feizollahi et al. \cite{feizollahi2017exact} establishes that the primal \eqref{eq: MIP} and the dual \eqref{eq: dual problem} problems have the same optimal objective. 

Our \emph{first method} solves \eqref{eq: dual problem} based on ALM:
\begin{align*}
	\text{Step 1:}\quad & \text{solve \eqref{eq: dual function} by subroutine AUSAL}: (x^{k+1},z^{k+1})\gets \texttt{AUSAL}(\lambda^k,\rho^k,\epsilon); \label{eq: AL iter1}\\
 	\text{Step 2:}\quad & \text{compute}~\lambda^{k+1},\rho^{k+1}.
\end{align*}
AUSAL stands for ``\textbf{A}lternating \textbf{U}pdate for the \textbf{S}harp \textbf{AL} function'', where ``sharp" refers to the $\ell_1$ penalty. AUSAL is itself an iterative algorithm, alternating between the updates of $x$ and $z$:
\begin{align*}
	\text{Step 1a:} \quad & \text{let $x^{t+1} \in \Argmin_{x\in X}\mathcal{L}(x,z^t,\lambda^k, \rho^k)$;}\\
 	\text{Step 1b:} \quad & \text{add a \texttt{ReverseNormCut}($x^{t+1},z^t$) to $z$-subproblem and solve for $z^{t+1}$}.
\end{align*} 	

The $x$-subproblems and $z$-subproblems are MILPs in $x$ and $z$, respectively. A \emph{reverse norm cut} is a nonconvex minorant, which we borrow from global Lipschitz minimization~\cite{mayne1984outer}. At each iteration, we add a cut to the $z$-subproblem, making the subproblem closer to \eqref{eq: dual function}, so that the AUSAL iterations asymptotically solve \eqref{eq: dual function}. When a certain optimality gap falls under a tolerance $\epsilon>0$, the iterations stop and return a pair $(x^{k+1},z^{k+1})$ to complete ALM's Step 1. {The pair  $(x^{k+1},z^{k+1})$ is then used to update $(\lambda^{k+1}, \rho^{k+1})$ in Step 2.} Overall, the first method uses double loops, and each outer iteration calls AUSAL in Steps 1a-1b. Note that the reverse norm cuts generated in one call to AUSAL become useless after $\lambda$ and $\rho$ change, so we cannot use them in the next outer iteration. Details of this method are given in Section \ref{Section: ALM}.

Our \emph{second method} uses a single loop and a different kind of cut, which stays valid across iterations. The method resembles ADMM, cycling through the updates of $x$, $z$, and Lagrange multipliers and penalty parameter:
\begin{subequations}\label{eq: highlevel admm}
 	\begin{align*}
 		\text{Step 1:} \quad & \text{let $x^{k+1} \in \Argmin_{x\in X}\mathcal{L}(x,z^k,\mu^k, \beta^k)$};\\
 		\text{Step 2:} \quad & \text{add an \texttt{ALCut}($x^{k+1}, z^k, \mu^k, \beta^k$) to $z$-subproblem and solve for $z^{k+1}$}; \\
 		\text{Step 3:} \quad & \text{compute } \mu^{k+1},\beta^{k+1}.
 	\end{align*}
\end{subequations}
The $x$-subproblems and $z$-subproblems are MILPs in $x$ and $z$, respectively. We add \textit{AL cuts} (despite its name, we use them only in this ADMM-like method) so that the $z$-subproblems' objectives approximate a certain value function with increasing accuracies. The dual pair in this method is denoted by $(\mu, \beta)$ instead of $(\lambda, \rho)$ in the first method. Details of the second method are given in Section \ref{Section: ADMM}.

Compared to existing ALM and ADMM methods for convex and nonlinear optimization, our new treatments include: adding cuts to $z$-subproblems, properly updating penalty parameters, and different convergence analyses. Next, we review related methods and compare them to ours.

\subsection{Related Works}
\subsubsection{Augmented Lagrangian Decomposition and Challenges}\label{section: challenges}
The classic ALM uses the squared penalty $\sigma(\cdot)=\frac{1}{2}\|\cdot\|^2_2$. Its convergence is well-established for convex programs \cite{rockafellar1973multiplier,rockafellar1976augmented} and smooth nonlinear programs \cite{andreani2007augmented,bertsekas2014constrained}. {We focus on works where $X$ and $Z$ contain discrete variables.}

With MILP \eqref{eq: MIP} and nonconvex nonsmooth problems, ALM algorithms work under the theory of nonconvex AL duality \cite{gasimov2002augmented,huang2003unified,rockafellar2009variational}. 
In particular, the theory states that $d(\lambda, \rho)$ in \eqref{eq: dual function} is concave and upper-semicontinuous in $(\lambda,\rho)$ \cite[Exercise 11.56]{rockafellar2009variational}, {so it can be maximized by a method for nonsmooth convex optimization. Specifically, works \cite{burachik2006modified,burachik2010primal,burachik2013inexact} introduce modified subgradient methods, and \cite{cordova2020revisiting} describes an inexact bundle method. When \eqref{eq: dual function} has a solution, dual convergence holds under proper conditions; convergence of the primal iterates can also be established. However, the challenge lies in solving \eqref{eq: dual function} with its nonconvex mixed-integer constraint sets $X$ or $Z$. All the methods in \cite{burachik2006modified,burachik2010primal,burachik2013inexact,cordova2020revisiting} assume a (nearly) exact oracle of \eqref{eq: dual function}, that is, finding a (nearly) global minimizers of \eqref{eq: dual function}. Paper \cite{cordova2020revisiting} proposes to alternatively update $x$ and $z$ variables, 
but this approach may get stuck at a non-stationary point due to the nonsmooth term $\|Ax+Bz\|_1$ in \eqref{eq: AL}; see \cite{xu2017globally} for an example. For the special case of \eqref{eq: MIP} with $B = I$, $g=0$, and $Z$ being a linear subspace, work \cite{boland2019parallelizable} applies a proximal ALM to its convex relaxation, but it cannot guarantee primal feasibility. In comparison, our first method ensures solving \eqref{eq: dual function} and thus finding a global solution to MILP \eqref{eq: MIP} under proper conditions.

\subsubsection{ADMM and Challenges}\label{section: challenges ADMM} 
Since ADMM updates $x$ and $z$ separately, it tends to have cheaper updates than ALM and, further thanks to the structure \eqref{eq: block}, ADMM generates smaller parallelizable $x_i$-updates. Although ADMM is structually similar to ALM, its convergence analysis is different and trickier than that of ALM.

ADMM and its variants under convexity have matured in the last decade~\cite{eckstein1992douglas,gabay2024applications,monteiro2013iteration}. Their analyses are based on assembling convex inequalities, which clearly fail to hold for MILP subproblems. There is little we can borrow there to develop our method. Recently, ADMMs for discrete and mixed-integer problems started to appear. Paper \cite{yao2019admm} reformulates a vehicle routing problem as a multi-block MILP and applies ADMM with the typical quadratic penalty. Thanks to their binary variables, their mixed-integer quadratic program (MIQP) subproblems reduce to MILPs. The formulations studied in \cite{alavian2017improving,kanno2018truss,kanno2018alternating,takapoui2017alternating,takapoui2020simple,yadav2016new} can all reduce to the form $\min_{x} \{f(x)~|~ x\in X \cap Z \}$, where $f$ and $X$ are continuous but $Z$ is a discrete set. Introduce variable $z$ for the reformulation $\min_{x,z} \{f(x)~|~ x = z, x\in X, z\in Z \}$, and we can apply ADMM. The $x$-subproblem is convex (except in \cite{kanno2018alternating}, where it is a nonlinear program), and the $z$-subproblem reduces to a projection onto a discrete set, which may admit a closed-form formula. These works report encouraging numerical results and discuss penalty parameter tuning and restart heuristics. Unfortunately, none of them comes with convergence guarantees to global solutions. In fact, there are examples on which ADMM either diverges or converges to local or infeasible solutions.

ADMMs for nonconvex continuous programs have also been studied \cite{jiang2019structured,melo2017iteration,sun2019two,wang2015global}. These works show convergence to stationary points for problems satisfying certain \emph{structural conditions}. For example, \cite{wang2015global} assumes 1) $\mathrm{Im}(A) \subseteq \mathrm{Im}(B)$, and 2) $Z = \R^m$ (they fail to hold for the above reformulation when $Z\neq\R^m$). Paper~\cite{wu2018ell} introduces an ADMM to solve problems involving binary constraints $x\in \{0,1\}^n$, which is reformulated as the intersection of $[0,1]^n$ and a shifted $\ell_p$-sphere. The resulting ADMM converges to a stationary point of a perturbed problem that satisfies the aforementioned two conditions. Compared to the existing ADMMs, our second method is the first attempt towards obtaining global solutions of MILPs.

\subsubsection{Two-stage stochastic MILP}\label{section: other literature}
Two-stage stochastic programs give rise to the block-angular structure \eqref{eq: block}. A recurring theme is Benders decomposition \cite{benders1962partitioning,van1969shaped}. In a typical iteration, given the first-stage variable $z$, the second-stage problem, or recourse problem, in $x$ is solved, and valid inequalities, or \textit{Benders cuts}, are generated to approximate the first stage's value function, $Q(z) = \min_{x\in X}\{c^\top x:~Ax = -Bz\}$. When the problem has continuous recourse, i.e., when $X$ is a polyhedron, Benders cuts are sufficient to guarantee convergence  \cite{kall1994stochastic} due to the convexity and finite piece-wise linearity of $Q$. When $X$ contains discrete decisions, $Q$ is not convex or continuous in general \cite{blair1977value}. Many works study binary $Z$~\cite{angulo2016improving, gade2014decomposition, laporte1993integer, ntaimo2013fenchel, sen2005c, sen2006decomposition, sherali2002modification}, and establish finite convergence. When $Z$ contains general mixed-integer variables, one usually needs to invoke a customized branch-and-cut algorithm for the first-stage problem \cite{ahmed2004finite, bodur2017strengthened, caroe1999dual, chen2011finite, chen2012computational, chen2022generating, qi2017ancestral} or introduce additional variables to formulate certain nonlinear approximation of the value function \cite{ahmed2019stochastic, caroe1998shaped}. The two approaches share some similarities. In theory, both linear and nonlinear cuts are trying to approximate the value function, and in practice, both approaches rely on the fast development of general-purpose MILP solvers and integration with modeling languages. The first one is usually implemented with solvers' callback functions, and the second one calls solvers a series of times. Our proposed methods take the second approach. Other approaches are also used to solve the first-stage problem, including heuristics, constraint program, and column generation; we refer interested readers to the survey~\cite{rahmaniani2017benders}.

Besides Benders cuts, there are cuts based on the Lagrangian duality for mixed-integer recourse. Zou et al. \cite{zou2019stochastic} propose Lagrangian cuts and strengthened Benders cuts for multistage stochastic programs with binary state variables. Chen and Luedtke \cite{chen2022generating} propose a normalization perspective 
for generating Lagrangian cuts and accelerate the cut generation process by focusing on a restricted subspace. Benders cuts and Lagrangian cuts have also been combined \cite{rahmaniani2020benders} and extended to solve mixed-integer nonlinear programs \cite{li2018improved, li2019generalized}. 

More recently, nonlinear cuts based on the augmented Lagrangian duality have been studied. Ahmed et al. \cite{ahmed2019stochastic} propose reverse norm cuts and AL cuts to approximate Lipschitz value functions. Reverse norm cuts date back to global Lipschitz minimization \cite{mayne1984outer}, and AL cuts are derived from exact penalization established by Feizollahi et al. \cite{feizollahi2017exact}. Later, Zhang and Sun \cite{zhang2019stochastic} use the exact penalization property as a workaround to the complete recourse condition; they further propose generalized conjugacy cuts, which unify reverse norm cuts and AL cuts, and study iteration complexities of stochastic dual dynamic programs (SDDP). 

Our proposed methods are motivated by \cite{ahmed2019stochastic, zhang2019stochastic}. 
The proposed AUSAL subroutine applies reverse norm cuts to the augmented Lagrangian function, and the proposed ADMM variant uses AL cuts. Our iteration complexity analysis is inspired by \cite{zhang2019stochastic}. We highlight some key differences and our contributions in the next subsection. 

Another related work \cite{van2020converging} uses \textit{scaled cuts}, a class of linear cuts derived by scaling AL cuts, within the Benders framework. The sequence of scaled cut closures converges to the convex envelope of the value function. Despite promising theoretical properties, computing scaled cuts is tricky, involving the combination of fixed-point iterations with a row generation scheme or cutting plane techniques. To speed up convergence, a heuristic is used to generate cuts that are only locally valid.

\subsection{Contributions}
Building on and extending the tools developed in \cite{ahmed2019stochastic, zhang2019stochastic}, we propose two algorithms for general two-block MILPs that take advantage of the structure \eqref{eq: block} with convergence guarantees to globally optimal solutions. 

{The first algorithm is an ALM where each subproblem is approximately solved by AUSAL, for which we give iteration complexity estimates. We further describe two approaches to obtain an approximate global solution of MILP \eqref{eq: MIP}:} (i) applying AUSAL to the penalty formulation, i.e., keeping $\lambda=0$ in \eqref{eq: dual function} and increasing the penalty parameter to reach exact penalty, and (ii) {using AUSAL as an oracle to solve the augmented Lagrangian dual problem and updating both Lagrange multipliers and the penalty parameter.} Under (ii), we provide two variants of the subgradient method.

Our second algorithm is a variant of ADMM and utilizes AL cuts. 
Works \cite{ahmed2019stochastic, zhang2019stochastic} both need the exact solutions to the augmented Lagrangian dual problem in $x$, which requires solving a sequence of augmented Lagrangian relaxations in every iteration. 
In contrast, we generate an AL cut by solving a \textit{single} augmented Lagrangian relaxation in variable $x$ in each iteration; see Section \ref{sec: al cuts} for more details. Our ADMM approach is similar to the strengthened augmented Benders cut 
briefly described in \cite[Section 4.1]{ahmed2019stochastic}; however, they do not update dual information, and their analysis cannot explain its convergence. We fill this gap by (i) giving conditions regarding the sequence of multipliers and penalty parameter that are sufficient for convergence, and (ii) establishing a finite convergence to an $\epsilon$-solution of MILP \eqref{eq: MIP}. Our analysis of ADMM using AL cuts appears to be new, and our method also generalizes ADMM from convex and nonlinear optimization to discrete optimization.

Both methods take advantage of \eqref{eq: block} to decompose the $x$-subproblem into smaller independent $x_i$-subproblems, which are amenable for parallel computing. We conduct numerical experiments on {three classes of MILPs to evaluate the practical performance of the proposed methods. Surprisingly, they exhibit advantages on structured problems over the state-of-the-art MILP solvers.}

A disadvantage of our methods is that the $z$-subproblem grows in size as more cuts are added, making it increasingly difficult to solve. This is a common issue in cut-based MILP methods. We believe, however, we can alleviate this issue by generating more efficient cuts, applying row generation, or using other means to solve the $z$-subproblem efficiently. {In fact, our numerical experiments suggest that this issue is manageable when the dimension of $Z$ is mild.} Note that we do not see the proposed methods as a replacement for MILP solvers but a means to scale them to larger MILP instances, especially for those with block separable substructures. 

\subsection{Notation and Organization}
We let $\R$, $\Z$, $\Q$, and $\N$ denote the sets of real, integer, rational, and natural numbers. We write $[P]=\{1,\cdots,P\}$. For a vector $x \in \R^n$, use $\|x\|_p$ as the $\ell_p$-norm of $x$ for $1\leq p\leq \infty$. The inner product of $x,y\in \R^n$ is denoted by $\langle x,y \rangle$ or $x^\top y$. 
For a matrix $A \in \R^{m\times n}$, $\|A\|_p$ denotes its induced (operator) norm, and $\mathrm{Im}(A)$ denotes its column space. We introduce  $\overline{B}_p(x;R) = \{y\in \R^n:~\|x-y\|_p\leq R\}$ and $D_p(X) =\sup\{\|x-y\|_p:~x,y\in X\}$. {We call a set \emph{mixed-integer-linear (MIL) representable} if it can be described by a finite number of mixed-integer variables and linear constraints; a function is MIL representable if its epigraph is MIL representable. }

We state our assumption on problem \eqref{eq: MIP} and provide a more detailed review of background materials in Section \ref{Section: Literature}. We introduce the proposed ALM framework in Section \ref{Section: ALM} and the ADMM variant in Section \ref{Section: ADMM}, together with their convergence results. In Section \ref{Section: Experiments}, we discuss implementation issues and present numerical experiments. Finally, we give some concluding remarks in Section \ref{Section: Conclusion}.

\section{Preliminaries}\label{Section: Literature}
\subsection{Assumptions and Approximate Solution}\label{sec: assumption}
 Throughout this paper, we make the following assumption on MILP \eqref{eq: MIP}. 
\begin{assumption}~\label{assumption: mip} 
Problem \eqref{eq: MIP} is feasible, and constraint sets $X$ and $Z$ in \eqref{eq: XZ MIP} are compact and mixed-integer-linear representable,  i.e., 
\begin{align}\label{eq: XZ MIP}
		X = \{x\in \R_+^{n_1} \times \Z_+^{n-n_1}:~ Ex= f\},\quad Z = \{z\in \R_+^{d_1} \times \Z_+^{d-d_1}:~ Gz= h\},
\end{align}
for rational matrices $(E, G)$ and vectors $(f,h)$ of proper dimensions.
\end{assumption}
We measure the accuracy of an approximate solution as follows.
\begin{definition}\label{def: approximate}
Let $\epsilon\geq 0$. We say $(x^*,z^*)$ is an $\epsilon$-solution of the MIP \eqref{eq: MIP} if  $x^*\in X$, $z^*\in Z$, $c^\top x^* + g^\top z^* \leq p^*+ \epsilon$, and $\|Ax^*+Bz^*\|_1\leq \epsilon.$
\end{definition}
Note that infeasibility is measured in the $\ell_1$-norm due to our analysis. Since $(x^*,z^*)$ may be infeasible, it is possible that $c^\top x^*+g^\top z^* < p^*$.  

\subsection{Exact Penalization}\label{Subsection: exact penalty}
Since the AL dual problem \eqref{eq: dual problem} has weak duality $\sup_{\lambda\in \R^m, \rho \geq 0} d(\lambda, \rho) \leq p^*$, two follow-up questions are: 1) how to obtain strong duality, i.e., for ``$=$'' to hold, and 2) how to obtain an optimal primal solution by solving the dual problem? Feizollahi et al. \cite{feizollahi2017exact} provides positive answers to both questions for MILP. Recall the definition of exact penalization. 
\begin{definition}[Exact penalization {\cite[Definition 11.60]{rockafellar2009variational}}] \label{def: exact penalization}
	A dual variable $\lambda\in\R^m$ is said to support exact penalization if, for all sufficiently large $\rho > 0$,
	$$\Argmin_{x\in X,z\in Z} \{c^\top x+ g^\top z:~Ax+Bz=0\}=\Argmin_{x\in X,z\in Z} \mathcal{L}(x,z,\lambda, \rho).$$
\end{definition}
We say a pair $(\lambda, \rho)$ supports exact penalization if the above equation holds. We simply say $\rho$ supports exact penalization when $(0,\rho)$ does so.
\begin{theorem}[Exact penalization for MILP \cite{feizollahi2017exact}] \label{thm: mip exact penaliztion}
	Under Assumption \ref{assumption: mip}, strong duality holds for \eqref{eq: dual problem}: $\sup_{\lambda\in \R^m, \rho \geq 0} d(\lambda, \rho) = p^*.$ For any $\lambda \in \R^m$, there exists $\underline{\rho}>0$ such that for every $\rho \in [\underline{\rho}, +\infty)$, $(\lambda, \rho)$ supports exact penalization.
\end{theorem}
The results proved in \cite{feizollahi2017exact} applies to a broader class of penalty functions, including all norms in $\R^m$. This paper focuses on the $\ell_1$-norm for component-wise separability. The proximal augmented Lagrangian where $\sigma(\cdot)= \frac{1}{2}\|\cdot\|_2^2$ does not support exact penalization; in general, no finite penalty $\rho$ can close the duality gap \cite[Proposition 7]{feizollahi2017exact}. A complete characterization of exact penalization is given as follows.
\begin{theorem}[Criterion for exact penalization {\cite[Theorem 11.61]{rockafellar2009variational}}]\label{thm: exact}
	Suppose Assumption \ref{assumption: mip} holds. The following statements are equivalent:
	\begin{enumerate}
		\item The pair $(\lambda,\rho)$ supports exact penalization.
		\item The pair $({\lambda}, {\rho})$ solves the dual problem \eqref{eq: dual problem}.
		\item There exists $r>0$ such that $p(u)\geq p(0) + \langle \lambda, u\rangle -\rho\|u\|_1$ for all $u\in \overline{B}_{1}(0;r)$,  where $p(u) := \min_{x,z}\{c^\top x + g^\top z~|~  Ax+Bz+u = 0, x\in X, z\in Z\}.$
	\end{enumerate}
\end{theorem} 

{\subsection{Nonconvex Cuts}
In this subsection, we review the main tools used in our algorithmic development, namely, reverse norm cuts and AL cuts. }

{\subsubsection{Global Lipschitz Minimization and Reverse Norm Cuts} \label{sec: reverse norm cut}
Consider
\begin{align} \label{eq: lip minimization}
    v^* = \min_{z \in Z} f(z) + Q(z),
\end{align}
where $f:\R^d \rightarrow \R$ is a simple function, e.g., a linear one, and $Q:\R^d \rightarrow \R$ is $L$-Lipschitz with respect to the $\ell_1$-norm over  the compact set $Z\subseteq \R^d$. In particular, $Q$ can be nonconvex and we only have access to its zero-order oracle, i.e., given $z\in Z$, we can evaluate $Q(z)$. Since $Q$ is $L$-Lipschitz, given $\bar{z}\in Z$, the following inequality
\begin{align}\label{def: reversed norm}
    Q(z) \geq r(z;\bar{z}) := Q(\bar{z}) -L\|z-\bar{z}\|_1
\end{align}
holds for all $z\in Z$. Due to the term ``$-\|z- \bar{z}\|_1$" in its right-hand side,  inequality \eqref{def: reversed norm} is called a reverse norm cut by \cite{ahmed2019stochastic}. The function $r(z;\bar{z})$ is a lower approximation of $Q$ that is tight at $\bar{z}$, i.e., $Q(\bar{z}) = r(\bar{z};\bar{z})$. Given a set $\bar{Z} \subset Z$, further consider 
\begin{align}\label{eq: lower approx}
	\underline{R}(z; \bar{Z}) := \sup \{r(z;\bar{z}):~\bar{z}\in \bar{Z}\},
\end{align}
which remains a lower approximation of $Q$ and gets closer to $Q$ as the set $\bar{Z}$ gets larger. If $\bar{Z}= Z$, we have $\underline{R}(z; \bar{Z})= Q(z)$ over all $z\in Z$. As a result, we can use $\underline{R}$ as an approximation of $Q$, and iteratively refine $\underline{R}$ by expanding $\bar{Z}$. This idea was firstly proposed by Mayne and Polak \cite{mayne1984outer} in the 80s (to our knowledge) and revisited recently \cite{ahmed2019stochastic, malherbe2017global}. We formalize the procedure in Algorithm \ref{alg: lip minimization}.
\begin{algorithm}[h]
	\caption{: Global Minimization by Reverse Norm Cuts}\label{alg: lip minimization}
	\begin{algorithmic}[1]
		\STATE \textbf{Input}: $(z^0, L, \epsilon) \in Z \times \R_{++} \times \R_+$; 
 		\STATE set $Z_0 \gets \{z^0\}$ and initialize $\texttt{UB} \gets +\infty$;
		\FOR{$k = 1,2,\cdots$}
		\STATE compute $z^k \in \Argmin_{z\in Z} f(z) + \underline{R}(z; Z_{k-1})$, and let $t^k \gets \underline{R}(z^k; Z_{k-1})$; \label{eq: lip minimization-z}
        \STATE $\texttt{UB}\gets \min \{\texttt{UB}, f(z^k) + Q(z^k) \}$;
        \IF{$\texttt{UB} - f(z^k) - t^k \leq \epsilon$}
        \STATE \textbf{return} $z^k$.
        \ENDIF
        \STATE $Z_k \gets Z_{k-1} \cup \{z_k\}$; 
		\ENDFOR
	\end{algorithmic}
\end{algorithm}
\begin{remark}
    We give some remarks regarding Algorithm \ref{alg: lip minimization}. 
    \begin{enumerate}
        \item The initial point $z^0 \in Z$ is required only for notation consistency. A pre-solve step where $\underline{R}$ is replaced by a finite lower bound of $Q$ over $Z$  yields $z^0$. 
        \item The minimal Lipschitz constant $L$ can be replaced by any over-estimators. 
        \item Quantities $\texttt{UB}$ and $f(z^k) + t^k$ keep track of upper and lower bounds of $v^*$, respectively. Another termination criteria is to check whether $Q(x^k) - t^k \leq \epsilon$. 
        \item The problem in line \ref{eq: lip minimization-z} can be cast as 
        \begin{align}
            \min_{z\in Z, t} \{f(z) + t:~ t \geq Q(z^j) - L\|z-z^j\|_1,~j = 0, \cdots, k-1\}. 
        \end{align}
        When $f$ and $Z$ are MIL representable, the above problem can be formulated as a MILP by introducing additional $k \times d$ binary variables and  $k \times 2d$ non-negative variables; e.g., see \cite{ahmed2019stochastic}. 
    \end{enumerate}
\end{remark}
The convergence of Algorithm \ref{alg: lip minimization} is stated in the next theorem.
\begin{theorem}[\cite{ahmed2019stochastic, mayne1984outer}]
    If $\epsilon > 0$, then Algorithm \ref{alg: lip minimization} terminates in a finite number of iterations with an $\epsilon$-optimal solution to problem \eqref{eq: lip minimization}, i.e., a point $z^k \in Z$ such that $f(z^k) + Q(z^k) \leq v^* + \epsilon$. If $\epsilon = 0$, then either Algorithm \ref{alg: lip minimization} terminates in a finite number of iterations with a global optimal solution to problem \eqref{eq: lip minimization}, or the sequence $\{f(z^k) + t^k\}_{k\in \N}$ converges to $v^*$ monotonically from below, and any limit point $z^*$ of $\{z^k\}_{k\in \N}$ is a global optimal solution to problem \eqref{eq: lip minimization}.
\end{theorem}
In addition, inspired by the analysis in \cite{zhang2019stochastic}, we further derive an iteration complexity estimate for Algorithm \ref{alg: lip minimization}, which complements the results in \cite{ahmed2019stochastic, mayne1984outer}. 
\begin{theorem}\label{thm: compexlity of lip minimization}
    Let $\epsilon >0$, and suppose $Z \subseteq \overline{B}_1(\bar{z}; R)$ for some $\bar{z}\in \R^d$ and radius $R>0$. Then Algorithm \ref{alg: lip minimization} terminates in no more than $(1 + 4 L R \epsilon^{-1})^d$ iterations. 
\end{theorem}
\begin{proof}
    See Appendix \ref{sec: proof of lip minimization complexity}. 
\end{proof}}
Theorem \ref{thm: compexlity of lip minimization} gives an upper bound on the number of reverse norm cuts to form a continuous and uniform approximation of $Q$ over an $\ell_1$-ball covering the feasible region $Z$. When $Z$ is much ``smaller" than the $\ell_1$-ball, i.e., when $Z$ contains integer or mixed-integer variables, the bound in Theorem 2.8 could potentially overestimate. For example, when $Z\subset\Z^d$, a simple grid search is adequate to solve the problem. The proposed ALM and ADMM inherit the idea behind Algorithm \ref{alg: lip minimization}, and, as we will see in Section \ref{sec: investment}, might indeed emulate such an enumeration behavior. On the other hand, we also observe instances where a full enumeration is not necessary within the proposed ALM and ADMM framework, e.g., see Table \ref{table: investment_ub10_T}.

{\subsection{Augmented Lagrangian Cuts}\label{sec: al cuts}
Augmented Lagrangian cuts were introduced in \cite{ahmed2019stochastic} and generalized in \cite{zhang2019stochastic}. In this subsection, we review and modify AL cuts in the context of problem \eqref{eq: MIP}. Consider the penalty formulation of problem \eqref{eq: MIP}:
\begin{align}\label{eq: penalty}
    \min_{z\in Z} g^\top z + R_{\rho}(z),
\end{align}
where $R_{\rho}:\R^d \rightarrow \R$ is the optimal value of a partial minimization with respect to $x$: 
\begin{align}\label{eq: R_rho}
    R_{\rho}(z) : = \min_{x\in X} c^\top x + \rho\|Ax+Bz\|_1.
\end{align}
Algorithm \ref{alg: lip minimization} can be applied to solve problem \eqref{eq: penalty} by generating reverse norm cuts of $R_{\rho}$; we elaborate this idea in Section \ref{Subsection: penalty}. We consider another class of nonconvex cuts, augmented Lagrangian (AL) cuts \cite{ahmed2019stochastic} that can be used to approximate $R_{\rho}$.
}

Let $\mu \in \R^m$ and $\beta\geq 0$, and define 
\begin{align}\label{eq: further_relax}
	P(z, \mu, \beta) := \min_x c^\top x + \langle \mu, Ax+Bz\rangle + \beta\|Ax+Bz\|_1.
\end{align}
Notice that 
\begin{align}\label{eq: weak duality}
	P(z, \mu, \beta) \leq & \min_x c^\top x + (\beta+ \|\mu\|_{\infty})\|Ax+Bz\|_1 \leq R_\rho(z)
\end{align}
for all $(\mu, \beta) \in \Lambda(\rho) := \{(\mu, \beta)\in \R^{m+1}~|~ \beta \geq 0, \beta + \|\mu\|_{\infty} \leq \rho\}.$
Inequality \eqref{eq: weak duality} is a \textit{weak duality} result in the sense that $P(z,\mu, \beta)$ provides a lower bound for $R_{\rho}(z)$ when the pair $(\mu,\beta)$ is constrained in $\Lambda(\rho)$. It turns out that \textit{strong duality} also holds: 
\begin{align}\label{eq: strong duality}
	R_{\rho}(z) =&  \max_{(\mu, \beta)\in \Lambda(\rho)} P(z, \mu, \beta) \\
	= &  \max_{(\mu, \beta)\in \Lambda(\rho)} \min_{x\in X} c^\top x + \langle \mu, Ax+Bz\rangle + \beta \|Ax+Bz\|_1, \notag 
\end{align}
simply due to the fact that $(0,\rho)\in \Lambda(\rho)$ and $P(z, 0, \rho) = R_{\rho}(z)$ by definition. Given $\bar{z}$ and $(\bar{\mu},\bar{\beta}) \in \Lambda(\rho)$, suppose we solve the problem \eqref{eq: further_relax}, then for any $z\in Z$,
\begin{align*}
	R_\rho(z) \geq   & P(z, \bar{\mu}, \bar{\beta}) =   \min_{x\in X} c^\top x + \langle \bar{\mu}, Ax+Bz\rangle + \bar{\beta}\|Ax+Bz\|_1\\
		 \geq   & \min_{x\in X} c^\top x + \langle \bar{\mu}, Ax+B\bar{z}\rangle + \bar{\beta}\|Ax+B\bar{z}\|_1 + \langle \bar{\mu}, Bz-B\bar{z}\rangle - \bar{\beta}\|Bz-B\bar{z}\|_1 \\
		  = & P(\bar{z}, \bar{\mu}, \bar{\beta}) + \langle \bar{\mu}, Bz-B\bar{z}\rangle - \bar{\beta}\|Bz-B\bar{z}\|_1. 
\end{align*}
Therefore, the function defined by 
\begin{align}\label{eq: generalized cut}
	\tilde{r}(z; \bar{z}, \bar{\mu}, \bar{\beta}):=P(\bar{z}, \bar{\mu}, \bar{\beta}) + \langle \bar{\mu}, Bz-B\bar{z}\rangle - \bar{\beta}\|Bz-B\bar{z}\|_1
\end{align}
is a lower approximation of $R_{\rho}(z)$.
\begin{definition}
	We call the inequality $R_\rho(z) \geq 	\tilde{r}(z; \bar{z}, \bar{\mu}, \bar{\beta})$ an augmented Lagrangian (AL) cut at $\bar{z}$ parameterized by $(\bar{\mu}, \bar{\beta})$. We say the cut is tight at $z$ if $R_\rho(z) = \tilde{r}(z; \bar{z}, \bar{\mu}, \bar{\beta})$.
\end{definition}
We note that an AL cut is not necessarily tight since $P(\bar{z}, \bar{\mu}, \bar{\beta})< R_{\rho}(\bar{z})$ when $(\bar{\mu},\bar{\beta}) \in \Lambda(\rho)$ is not optimal for \eqref{eq: strong duality}. The additional linear term $\langle \bar{\mu}, Bz-B\bar{z}\rangle$ corresponds to a rotation around the pivot $(\bar{z}, P(\bar{z}, \bar{\mu}, \bar{\beta})) \in \R^{d+1}$. In addition, since $\|\bar{\mu}\|_{\infty} + \bar{\beta} \leq \rho$, the AL cut may have a smaller Lipschitz constant than $R_{\rho}$. Geometrically, the rotation effect and smaller Lipschitz constant allow an AL cut to be ``fatter" than $R_{\rho}$ and thus covers a wider range in $Z$ than a reverse norm cut. Moreover, a smaller value of $\bar{\beta} + \|\bar{\mu}\|_{\infty}$ can be preferable for optimization solvers.
		
Both algorithms in \cite{ahmed2019stochastic,zhang2019stochastic} assume the AL cut generated in each iteration is tight in the sense that a pair $(\mu, \beta)$ optimal to the maximization problem in \eqref{eq: strong duality} is available. In practice, one needs to call a subgradient method to solve the max-min problem \eqref{eq: strong duality} in a double-looped fashion, and hence in each iteration, multiple MILPs in the form of \eqref{eq: further_relax} needs to be solved until convergence. In contrast, the proposed ADMM in Section \ref{Section: ADMM} generates an AL cut in iteration $k$ by solving a \textit{single} problem \eqref{eq: further_relax} with $(\mu, \beta) = (\mu^k, \beta^k)$, and guide convergence through proper updates of $(\mu^{k+1}, \beta^{k+1})$. We note that the strengthened augmented Benders cut in \cite{ahmed2019stochastic} is also generated by solving a single MILP in $x$ and is implemented to accelerate convergence of multistage SDDP. In the context of SIP, we provide theoretical justification to this computationally favorable scheme in the two-stage case. Compared to the ADMM literature, our method finds global solutions of nonconvex problems with convergence guarantees.

\section{An ALM Method empowered by AUSAL}\label{Section: ALM}
In this section, we introduce an ALM framework for MILP \eqref{eq: MIP}. In Section \ref{Subsection: primal}, we present the AUSAL algorithm for subproblem \eqref{eq: dual function}. In order to find an $\epsilon$-solution of MILP \eqref{eq: MIP}, AUSAL is further applied to the penalty formulation in Section \ref{Subsection: penalty} or embedded in the ALM in Section \ref{Subsection: dual}. We present two variants of ALM based on different subgradient updates. 

\subsection{AUSAL}\label{Subsection: primal}
Consider the augmented Lagrangian relaxation \eqref{eq: dual function}:
\begin{equation*}
 d(\lambda, \rho) = \min_{x\in X, z\in Z} ~c^\top x+ g^\top z +\langle \lambda, Ax+Bz\rangle + \rho\|Ax+Bz\|_1,
 \end{equation*}
where $\lambda \in \R^m$ and $\rho$ are fixed as constants. We decompose the minimization into two stages: the inner stage minimizes over $x\in X$ with $z$ fixed, and the outer stage minimizes over $Z$: 
\begin{align}
    \label{eq: value function}
 	& R(z) =R(z;\lambda, \rho) := \min_{x\in X}~ \langle c+ A^\top \lambda,x \rangle + \rho\|Ax+Bz\|_1,\\
	\label{eq: primal subproblem z}
 	&d(\lambda,\rho)= \min_{z\in Z} ~\langle g+B^\top \lambda, z\rangle + R(z).
\end{align}
Note that $R(z)$ is well defined over $z\in \R^d$ due to the compactness of $Z$. The function $R(z)$ is known as the value function or cost-to-go function in the context of sequential decision making. We omit the dependency on $(\lambda,\rho)$ for simplicity in this section. The next lemma suggests that $R$ is Lipschitz. 

\begin{lemma}\label{lemma: value function}
	Suppose Assumption \ref{assumption: mip} holds, and $X$ is compact for any right-hand side vector $f$ (possibly empty).
	Then $R(z)$ is piecewise-linear and $L_\rho$-Lipschitz continuous with respect to the $\ell_1$-norm over $\R^d$, where $L_\rho := \rho\|B\|_1$.
\end{lemma}
\begin{proof}
We show $R(z)$ is piecewise linear by considering the standard-form MILP: 
\begin{align*}
	v(b) := \min_{x}\left\{c^\top x \Big|  Ax := \begin{bmatrix} A_1x \\A_2x \end{bmatrix} 	=\begin{bmatrix} b_1 \\b_2\end{bmatrix} = :b, ~x\in \R^{p}_+ \times \Z^{q}_+ \right\}.
\end{align*}
Let $v(b) = +\infty$ if it is infeasible and $-\infty$ if it is unbounded. Define $D := \{b~|~v(b) < +\infty\}$. It is shown in \cite{blair1977value} that if the MILP is described by rational data and $v(b)>-\infty$ for all $b\in D$, then $v(b)$ is a piecewise linear function over $D$. Now fix $b_2\in \mathrm{Im}(A_2)$ and define $v_1(b_1) = v([b_1^\top, b_2^\top]^\top)$. Then $v_1(b_1)$ is also piecewise linear over its domain $D_1 := \{b_1~|~[b_1^\top, b_2^\top]^\top \in D\}.$
Notice that the premise of \cite{blair1977value} is satisfied for problem \eqref{eq: value function}, and we can equivalently write $R(z) = \min_{x\in X, u}  \{\langle c+ A^\top \lambda, x\rangle +  \rho\|Ax+Bu\|_1~|~u = z\}$, where $z$ plays the role of $b_1$. Since the rest of the problem defining $R(z)$ can be cast as a standard MILP, we conclude that $R(z)$ is a piecewise linear function in $\R^d$.

To prove $R(z)$ is Lipschitz, take any $z^1, z^2 \in \R^d$, and let $x^1, x^2$ be the optimal solution to \eqref{eq: value function} with $z=z^1$ and $z=z^2$, respectively. It holds that $R(z^1) - R(z^2) \leq \langle c+ A^\top \lambda,x^2 \rangle  + \rho \|Ax^2+Bz^1\|_1-\langle c+ A^\top \lambda,x^2\rangle  - \rho \|Ax^2+Bz^2\|_1\leq \rho \|B\|_1\|z^1 - z^2\|_1$, where the first inequality is due to the optimality of $x^1$ in the definition of $R(z^1)$, and the second inequality is by the triangle inequality. Similarly we have $R(z^2) - R(z^1)\leq \rho \|B\|_1\|z^1-z^2\|_1$, which concludes the proof with the claimed modulus $L_\rho$. 
\end{proof}

{
As a result, we can apply Algorithm \ref{alg: lip minimization} to solve problem \eqref{eq: dual function}, and we name the resulting scheme Alternating Update for the Sharp Augmented Lagrangian function, dubbed AUSAL. See Algorithm \ref{alg1}. 
\begin{algorithm}[h]
	\caption{: AUSAL}\label{alg1}
	\begin{algorithmic}[1]
		\STATE \textbf{Input} $(\lambda, \rho, \epsilon) \in \R^m \times \R_{++} \times \R_+$;
		\STATE initialize $f(\cdot) \gets \langle g + B^\top \lambda, \cdot \rangle$, $Q(\cdot)\gets R(\cdot)$, $z_0 \in Z$, and 
		$L_{\rho}\gets \rho \|B\|_1$; 
		\STATE compute $z^*\in Z$ by calling Algorithm \ref{alg: lip minimization} with input $(z_0, L_\rho, \epsilon)$;
		\STATE \textbf{return} $(x^*, z^*)$ where $x^*\in X$ is a minimizer in \eqref{eq: value function} with $z = z^*$. \label{eq: ausal-return}
	\end{algorithmic}
\end{algorithm}
\begin{remark}
    In view of the block-angular structure \eqref{eq: block}, the evaluation of $R(z)$ can be decomposed into $P$ parallel subproblems thanks to the component-wise separability of $\ell_1$-norm. 
\end{remark}
}
{
We immediately have the following corollary of Theorem \ref{thm: compexlity of lip minimization}. 
\begin{corollary}\label{thm: primal convergence}
	Suppose Assumption \ref{assumption: mip} holds, $\epsilon>0$, and $Z \subseteq \overline{B}_1(\bar{z}; R)$ for some $\bar{z}\in \R^d$ and $R>0$. AUSAL terminates in no more than $(1+4\rho \|B\|_1 R\epsilon^{-1})^d=\mathcal{O}(\rho^d \epsilon^{-d})$ iterations.
\end{corollary}
}
\begin{remark}
	We do not expect this upper bound to be practically informative since it depends exponentially on the dimension of $Z$. The results indicates that ASUAL is a parallelizable algorithm with finite convergence to some $\epsilon$-solution. 
\end{remark}

\subsection{Penalty Approach}\label{Subsection: penalty}
We present a penalty method that solves the original problem \eqref{eq: MIP} using AUSAL as a subroutine. 
{We first present a helpful lemma.
\begin{lemma}\label{lemma: penalty}
    Let $(\lambda, \rho) \in \R^m \times \R_{++}$ satisfy $\rho \geq  \|\lambda\|_\infty$, and $(\bar{x}, \bar{z})$ be an $\epsilon$-solution to the augmented Lagrangian relaxation $\min_{x\in X, z\in Z} \mathcal{L}(x, y, \lambda, \rho)$. Then we have the following chain of inequalities: 
    \begin{align}\label{eq: penalty key ineq}
     c^\top \bar{x} +g^\top \bar{z} \leq & c^\top \bar{x} +g^\top \bar{z}  + (\rho-\|\lambda\|_{\infty})\|A\bar{x}+B\bar{z}\|_1  \\
     \leq & c^\top \bar{x} +g^\top \bar{z}  +  \langle \lambda, A\bar{x}+B\bar{z} \rangle + \rho\|A\bar{x}+B\bar{z}\|_1 \leq p^* +\epsilon. \notag 
\end{align}
\end{lemma}
\begin{proof}
The inequalities are due to $\rho \geq \|\lambda\|_{\infty}$, H\"{o}lder's inequality applied to $ \langle \lambda, A\bar{x}+B\bar{z} \rangle$, and $(\bar{x}, \bar{z})$ being $\epsilon$-optimal for \eqref{eq: dual function} and $d(\lambda, \rho)\leq p^*$, respectively. 
\end{proof} }
Suppose that we have a pair $(\lambda, \rho)$ that supports exact penalization; then we only need to call AUSAL to solve the augmented Lagrangian relaxation \eqref{eq: dual function} once. Our first result suggests that, with such a pair $(\lambda, \rho)$, we can expect to obtain an $\epsilon$-solution of \eqref{eq: MIP} upon the termination of AUSAL.
\begin{theorem}\label{thm: exact_1}
	Suppose $(\lambda, \rho)\in \R^{m+1}$ satisfies that 1) $\rho\geq \|\lambda\|_\infty$, and 2) $(\lambda, \rho-1)$ supports exact penalization. Then AUSAL applied to the augmented Lagrangian relaxation \eqref{eq: dual function} returns an $\epsilon$-solution of MILP \eqref{eq: MIP}.
\end{theorem}
\begin{proof}
	Let $(\bar{x}, \bar{z})$ denote the solution returned by Algorithm \ref{alg1}, which is $\epsilon$-optional to the augmented Lagrangian relaxation $\min_{x\in X, z\in Z} \mathcal{L}(x, y, \lambda, \rho)$. By Lemma \ref{lemma: penalty}, we have $c^\top \bar{x} +g^\top \bar{z}  \leq p + \epsilon$. Furthermore, the last inequality in \eqref{eq: penalty key ineq} implies that $\rho\|A\bar{x}+B\bar{z}\|_1 \leq  p^* + \epsilon - c^\top \bar{x}  -g^\top \bar{z}  -\langle \lambda,  A\bar{x}+B\bar{z}\rangle;$ since $(\lambda, \rho-1)$ also supports exact penalization, it holds that $p^* \leq c^\top \bar{x} +g^\top \bar{z}  + \langle \lambda, A\bar{x}+B\bar{z}\rangle + (\rho-1)\|A\bar{x}+B\bar{z}\|_1.$ The above two inequalities together imply $ {\|A\bar{x}+B\bar{z}\|_1} \leq \epsilon$. 
\end{proof}
Since the current $(\lambda,\rho)$ may not support exact penalization, we must address the question: given any $\lambda\in \R^m$, what is the sufficiently large $\rho$ for AUSAL to deliver an $\epsilon$-solution of MILP \eqref{eq: MIP}. Note that $\lambda=0$ is possible, and if this is the case, the primal subproblem \eqref{eq: dual function} becomes the penalty formulation of \eqref{eq: MIP}.
\begin{theorem}\label{thm: penalized}
	Let $\lambda\in \R^m$ and $\epsilon >0$. Then AUSAL applied to the primal subproblem \eqref{eq: dual function} with
	$\rho = \left( {\|c\|_{1} D_\infty(X) + \|g\|_1D_\infty(Z)}\right)\epsilon^{-1} + \|\lambda\|_\infty + 1$
	returns an $\epsilon$-solution $(\bar{x}, \bar{z})$ of the MILP \eqref{eq: MIP}. {In particular, under assumptions of Corollary \ref{thm: primal convergence}, AUSAL terminates in at most $(1+ 4\rho \|B\|_1R\epsilon^{-1})^d = \mathcal{O}(\epsilon^{-2d})$ iterations.}
\end{theorem}
\begin{proof}
	By Lemma \ref{lemma: penalty}, $(\rho - \|\lambda\|_{\infty}) \|A\bar{x}+B\bar{z}\|_1 \leq p^* +\epsilon - c^\top \bar{x} - g^\top \bar{z}$ and $c^\top \bar{x} +g^\top \bar{z}  \leq p + \epsilon$.
    Let $(x^*, z^*)$ be an optimal solution of MILP \eqref{eq: MIP}, then  
	\begin{align*}
		 \|A\bar{x}+B\bar{z}\|_1 \leq \frac{c^\top (x^*-\bar{x}) + g^\top(z^*-\bar{z})+\epsilon}{\rho-\|\lambda\|_\infty}	\leq \frac{ {\|c\|_{1}D_\infty(X) + \|g\|_{1}D_\infty(Z)} + \epsilon}{\rho-\|\lambda\|_\infty}.
	\end{align*}
 	It is straightforward to verify the choice of $\rho$ ensures $\|Ax^* + Bz^*\|_1\leq \epsilon$, and finally the complexity result is a direct consequence of Corollary \ref{thm: primal convergence}. 
\end{proof}

The choice of $\rho$ in Theorem \ref{thm: penalized} serves as a sufficient condition for AUSAL to deliver an $\epsilon$-solution of \eqref{eq: MIP} given any $\lambda\in \R^m$. However, we observed that a large penalty $\rho = \mathcal{O}(\epsilon^{-1})$ numerically slows down the algorithm, even in its early stages. Hence in practice, instead of calling AUSAL just once with some fixed $(\lambda, \rho)$, one can also update them in an iterative fashion, which we discuss in the next section.

\subsection{ALM and Dual Updates}\label{Subsection: dual}
We present an ALM that uses AUSAL for its subproblem and two different subgradient methods for the update of $(\lambda,\rho)$.
This method is appropriate when a pair $(\lambda,\rho)$ that supports exact penalization is not known initially.
The dual function $d(\lambda, \rho)$ is a concave and upper-semicontinuous function in $(\lambda, \rho)$. Let $(\bar{x},\bar{z})$ be a pair of solution returned by AUSAL with input $(\bar{\lambda}, \bar{\rho}, \epsilon)$. 
Then for all $(\lambda, \rho)\in \R^{m} \times \R_+$, it is straightforward to see that
\begin{align*}
	(-d)(\lambda, \rho) \geq (-d)(\bar{\lambda}, \bar{\rho}) -
	 \begin{bmatrix}
 		 A\bar{x}+B\bar{z}\\
		 \| A\bar{x}+B\bar{z}\|_1 
 	\end{bmatrix}
 	^\top
 	\begin{bmatrix}
 		\lambda - \bar{\lambda}\\
		\rho -\bar{\rho} 
 	\end{bmatrix}
 	-\epsilon,
\end{align*}
or equivalently, 
\begin{align}\label{eq: dual_grad}
	-\begin{bmatrix}
		A\bar{x} + B\bar{z}\\
		 \|A\bar{x}+B\bar{z}\|_1 
	\end{bmatrix}
	\in \partial_{\epsilon} (-d) (\bar{\lambda}, \bar{\rho}).
\end{align}
When $\epsilon=0$, the above inclusion yields a convex subgradient. Therefore, we can use the primal information $(\bar{x}, \bar{z})$ to construct an $\epsilon$-subgradient of $(-d)$ at $(\bar{\lambda}, \bar{\rho})$ and apply an inexact subgradient method to solve the dual problem \eqref{eq: dual problem}. 
The compactness on $X$ and $Z$ ensures that $d(\lambda,\rho)$ is well defined for all $\lambda\in \R^m$ and $\rho \geq 0$; the exact penalization of MILP from Theorem \ref{thm: mip exact penaliztion} guarantees that $\Argmax d(\lambda, \rho)$ is nonempty. 

The subgradient method has many variants, such as the exactness of subgradient, choices of step size, and stopping criteria. Below we provide two specific implementations. Since the analysis is standard, we defer the proof in Appendix \ref{appendix: proof} only for completeness. Beyond the two implementations below, we can apply methods \cite{burachik2006modified,burachik2010primal,burachik2013inexact,cordova2020revisiting} with AUSAL as a subroutine to solve the dual problem \eqref{eq: dual problem}.

\subsubsection{First Subgradient Method and Iteration Complexity}
\begin{algorithm}[!ht]
	\caption{: A Subgradient Variant with Iteration Estimate on Objective Gap}\label{alg2}
	\begin{algorithmic}[1]
		\STATE \textbf{Input} $\epsilon_p\geq 0$;
		\STATE initialize $(\lambda^1, \rho^1)$, and a sequence $\{\tau_k\}_{k\in \N}$ such that $\tau_k > 0$, $\tau_k\rightarrow 0$, and  $\sum_{k\in \N}\tau_k = +\infty$ ;
		\FOR{$k = 1,2,\cdots$}
		\STATE $(x^k, z^k) \gets$ AUSAL$(\lambda^k, \rho^k, \epsilon_p)$;
		\STATE set $\alpha_k = \tau_k/(\sqrt{2}\|Ax^k+Bz^k\|_1)$;\label{alg2: step}
		\STATE $\lambda^{k+1}\gets \lambda^k + \alpha_k (Ax^k+Bz^k)$, $\rho^{k+1}\gets \rho^k + \alpha_k \|Ax^k+Bz^k\|_1$;\label{alg2: actual step}
		\ENDFOR
	\end{algorithmic}
\end{algorithm}
The first subgradient algorithm is presented as Algorithm \ref{alg2}. The constant $\epsilon_p\geq0$ controls the exactness of the subgradient in \eqref{eq: dual_grad}. If AUSAL returns $(x^k,z^k)$ with $Ax^k+Bz^k=0$, $(x^k,z^k)$ is an $\epsilon_p$-solution to MILP \eqref{eq: MIP}. To study convergence and complexity, we assume $\|Ax^k+Bz^k\|_1>0$ for all $k\in \N$, so that Algorithm \ref{alg2} will produce an infinite sequence $\{(\lambda^k,\rho^k)\}_{k\in \N}$. 
We introduce the following constants for our analysis:
\begin{align*}
	D^* := \Argmax_{\lambda,\rho\geq 0} d(\lambda,\rho), d_0 :=  \min_{(\lambda,\rho) \in D^*} \|\lambda^1-\lambda\|_1+|\rho^1-\rho|, M := & \max_{x\in X, z\in Z}\|Ax+Bz\|_1.
\end{align*}
Consider the step size $\alpha_k$ chosen as in line \ref{alg2: step} of Algorithm \ref{alg2}. Other standard choices include $\alpha_k =  \epsilon_d/(2M^2)$ or $\alpha_k = \epsilon_d/(2\|Ax^k+Bz^k\|_1^2)$ for some tolerance $\epsilon_d > 0$; their convergence results are similar and thus omitted.
\begin{theorem}\label{thm: dual complexity}
	The following statements hold.
	\begin{enumerate}
		\item Let $\epsilon_d > 0$. Suppose Algorithm \ref{alg2} performs $K = \lceil 2M^2d_0^2/\epsilon^2_d\rceil$ iterations with $\tau_k = d_0/\sqrt{K}$ for all $k\in [K]$. It holds that $\min_{k\in [K]} p^* - d(\lambda^k,\rho^k) \leq \epsilon_p + \epsilon_d.$
		\item Suppose $\epsilon_p = 0$, and the sequence $\{\tau_k\}_{k\in \N}$ also satisfies $0 < \tau_k\leq \tau$ for all $k\in \N$ and $\sum_{k\in \N}\tau^2_k < +\infty$. Then $(\lambda^k,\rho^k)$ converges to some $(\lambda^*,\rho^*)\in D^*$. 
	\end{enumerate}
\end{theorem}
\begin{proof}
	See Appendix \ref{sec: proof of dual compex}.
\end{proof}
\begin{remark}
	We give some remarks regarding the first case of Theorem \ref{thm: dual complexity}.
\begin{enumerate}
	\item The outer-level subgradient method can be preferred over solving a single penalty problem with a large $\rho = \mathcal{O}(\epsilon^{-1})$ suggested in Theorem \ref{thm: penalized} when $\epsilon_d$ is allowed to be larger than $\epsilon_p$ . Suppose $\rho^1 = \frac{d_0}{\sqrt{2K}}$ so that $\rho^k =\frac{kd_0}{\sqrt{2K}}$ for all $k\in [K]$.
	By Theorem \ref{thm: primal convergence} and the fact that $K = \mathcal{O}(\epsilon_d^{-2})$, the total number of AUSAL iterations required in Algorithm \ref{alg2} can be bounded by 
	\begin{align}\label{eq: alm_complexity}
		\sum_{k=1}^K \mathcal{O}\left[  \left( \frac{\rho^k}{\epsilon_p} \right)^d \right] 
		= \mathcal{O}\left[\epsilon_p^{-d}\left( \frac{1}{d+1} K^{d/2+1} + K^{d/2}\right)\right]= \mathcal{O}\left( \epsilon_p^{-d} \epsilon_d^{-d-2} \right),
	\end{align}
	where the first equality is due to 
	\begin{align*}
		\sum_{k=1}^K \left( \frac{k}{\sqrt{K}}\right)^d \leq K^{-d/2} \left(\int_1^{K} x^d ~ dx  + K^d \right)\leq  \frac{1}{d+1} K^{d/2+1} + K^{d/2}.	
	\end{align*}
	The $\mathcal{O}$-notation hides some constant dependent on $d$. Suppose $\epsilon_p = \epsilon \in (0,1)$, then
	\eqref{eq: alm_complexity} is no worse than $\mathcal{O}(\epsilon^{-2d})$ presented in Theorem \ref{thm: penalized} as long
	as $\epsilon_d \geq \epsilon^{\frac{d}{d+2}}$. For example, choosing $\epsilon_d = \epsilon^{\frac{d}{2d+4}}$ reduces \eqref{eq: alm_complexity} to $\mathcal{O}(\epsilon^{-1.5d})$. 
	\item If Algorithm \ref{alg2} does not find the optimal dual variable in $K$ iterations, then by part one of Theorem \ref{thm: dual complexity}, we can estimate the duality gap in the objective, and expect the best-so-far iterate $(\lambda^k, \rho^k)$ to be close to some optimal dual solution $(\lambda^*,\rho^*)$. Consequently, we can post-process to recover an optimal dual solution as shown in the next corollary.
\end{enumerate}
\end{remark}

\begin{corollary}
	Suppose Algorithm \ref{alg2} generates a pair $(\lambda^k, \rho^k)$ that satisfies $\|(\lambda^k, \rho^k)- (\lambda^*,\rho^*)\|_{\infty}\leq l$ for some $(\lambda^*,\rho^*)\in D^*$. Then $(\lambda^k, \rho^k + 2l)$ supports exact penalization. Applying AUSAL with $\lambda = \lambda^k$,  $\rho=\max\{\|\lambda_k\|_\infty, \rho^k + 2l+1\}$, and $\epsilon>0$ returns an $\epsilon$-solution of the MILP \eqref{eq: MIP}.
\end{corollary}
\begin{proof}
	The pair $(\lambda^k, \rho^k + 2l)$ supports exact penalization since, by Theorem \ref{thm: exact}, there exists $r>0$ such that for all $u\in \overline{B}_1(0;r)$, it holds that 
	\begin{align*}
		p(u) \geq & p(0)+ \langle \lambda^*, u\rangle - \rho^*\|u\|_1\\
			\geq &  p(0) + \langle \lambda^k, u\rangle - \rho^k\|u\|_1 - \|\lambda^k-\lambda^*\|_{\infty}\|u\|_1 - |\rho^k-\rho^*|\|u\|_1\\
		\geq &  p(0) + \langle \lambda^k, u\rangle - (\rho^k+2l)\|u\|_1,
	\end{align*}
    {where the third inequality is due to $\max\{ \|\lambda^k- \lambda^*\|_{\infty},  \|\rho^k- \rho^*\|_{\infty}  \} \leq l.$} The second claim follows directly from Theorem \ref{thm: exact_1}. 
\end{proof}

\subsubsection{Second Subgradient Method with Finite Convergence to an Approximate Solution}
The subgradient method proposed in the previous subsection maximizes the augmented Lagrangian dual function $d(\lambda, \rho)$ using inexact subgradients of $-d$; however, an approximate global solution to the primal problem \eqref{eq: MIP} may not be readily available from Algorithm \ref{alg2}, and some post-processing step needs to be invoked.
 In this subsection, we present the second subgradient variant in Algorithm \ref{alg3}, which directly returns an  $\epsilon_p$-solution in a finite number of calls of AUSAL.
\begin{algorithm}[!ht]
	\caption{: A Subgradient Variant with Finite Convergence}\label{alg3}
	\begin{algorithmic}[1]
		\STATE \textbf{Input} $\epsilon_p > 0$;
        \STATE initialize $(\lambda^1, \rho^1)$ with $\rho^1\geq \|\lambda^1\|_\infty$,  and some $\tau>0$ ;
		\FOR{$k = 1,2,\cdots$}
		\STATE $(x^k, z^k) \gets$ AUSAL$(\lambda^k, \rho^k, \epsilon_p)$;
		\IF {$\|Ax^k+Bz^k\|_1 \leq \epsilon_p$}\label{alg1: stop}
		\STATE \textbf{return} $(x^k, z^k)$.
		\ENDIF
		\STATE set $\alpha_k = \tau/\|Ax^k+Bz^k\|_1$;
		\STATE $\lambda^{k+1} \gets \lambda^k + \alpha_k (Ax^k+Bz^k)$, $\rho^{k+1} \gets \max\{\|\lambda^{k+1}\|_{\infty}, \rho^k + \alpha_k\|Ax^{k}+Bz^{k}\|_1\}$;
		\ENDFOR
	\end{algorithmic}
\end{algorithm}

If Algorithm \ref{alg3} terminates with $(x^k, z^k)$, then $x^k \in X$, $z^k\in Z$, and $\|Ax^k+Bz^k\|_1\leq \epsilon_p$. Since $\rho^k\geq \|\lambda^k\|_{\infty}$, lemma \ref{lemma: penalty} ensures $c^\top x^k + g^\top z^k \leq p^*+\epsilon_p$. Therefore, $(x^k,z^k)$ is indeed an $\epsilon_p$-solution of the MILP \eqref{eq: MIP}.

\begin{theorem}\label{thm: finite terminate}
	Let $\epsilon_p>0$. Algorithm \ref{alg3} returns an $\epsilon_p$-solution of MILP \eqref{eq: MIP} in a finite number of iterations.
\end{theorem}
\begin{proof}
	See Appendix \ref{sec: proof of finite term}. 
\end{proof}

\section{An ADMM-Based Method}\label{Section: ADMM}
In this section, we present a variant of ADMM that uses AL cuts introduced in Section \ref{sec: al cuts}. In iteration $k$, the method first approximately evaluates $R_{\rho}(z^{k-1})$ by solving the augmented Lagrangian relaxation \eqref{eq: further_relax} with some fixed $(\mu^k, \beta^k)$. Then an AL cut is generated to the $z$-subproblem where $z^k$ is computed. Finally, we update $(\mu^{k+1}, \beta^{k+1})$ and proceed to iteration $k+1$.  Recall the AL cut defined in \eqref{eq: generalized cut}: $\tilde{r}(z; \bar{z}, \bar{\mu}, \bar{\beta}):=P(\bar{z}, \bar{\mu}, \bar{\beta}) + \langle \bar{\mu}, Bz-B\bar{z}\rangle - \bar{\beta}\|Bz-B\bar{z}\|_1.$ A conceptual ADMM is described in Algorithm \ref{alg: admm}.
\begin{algorithm}[!h]
	\caption{: An ADMM Framework using AL Cuts} \label{alg: admm}
	\begin{algorithmic}[1]
		\STATE \textbf{Initialize} $(z^0, \mu^1, \beta^1)  \in  Z \times \R^m \times \R_{++}$;
		\FOR{$k = 1,2,\cdots$}
		\STATE compute  
		\begin{align}\label{eq: admm-x-subproblem}
		    x^k \in \Argmin_{x\in X} c^\top x + \langle \mu^k, Ax+Bz^{k-1}\rangle +\beta^k \|Ax+Bz^{k-1}\|_1;
		\end{align}
        \vspace{-0.2in}
		\STATE compute 
		 \begin{align}\label{eq: admm-z-subproblem}
		     (z^k, t^k) \in \Argmin_{z\in Z, t\in \R} \{g^\top z + t:~ t\geq \tilde{r}(z; z^{j-1}, \mu^j,\beta^j)~\forall j\in[k]  \};
		 \end{align}
         \vspace{-0.2in}
		\STATE update $(\mu^{k+1}, \beta^{k+1}) \in \R^m \times \R_{++}$;\label{line:admm_dual}
		\ENDFOR
	\end{algorithmic}
\end{algorithm}
\begin{remark}
    We give some remarks regarding the ADMM variant. 
    \begin{enumerate}
        \item The $x$-subproblem \eqref{eq: admm-x-subproblem} has the same form as in AUSAL, and can be decomposed into $P$ parallel smaller MILPs in view of the block structure \eqref{eq: block}.
        \item We do not specify how $(\mu^{k+1}, \beta^{k+1})$ is updated in Algorithm \ref{alg: admm}. Instead, we present a set of assumptions on the selection of $(\mu^{k+1}, \beta^{k+1})$ to establish convergence results.  Any updates that meet the assumptions will {ensure convergence}. We provide {specific examples and} some geometric intuitions in Section \ref{sec: admm asummption}.
    \end{enumerate}
\end{remark}

One major difference between our ADMM variant and the classic ADMM lies in the $z$-subproblem. In a traditional ADMM framework, the $z$-subproblem has the following structure:  
\begin{align}\label{eq: z-classic_admm}
	\min_{z\in Z} ~g^\top z + \left(c^\top x^k+ \langle \mu^k, Ax^{k}+Bz\rangle + \beta^k \sigma(Ax^k+Bz)\right),
\end{align}
where $\sigma(\cdot) = \frac{1}{2}\|\cdot\|_2^2$ (proximal Lagrangian) or $\sigma(\cdot) = \|\cdot\|_1$ (sharp Lagrangian) is usually used. {Recall the original problem is equivalent to $\min_{x\in Z} g^\top z + R_{\rho}(z)$ for some sufficiently large $\rho$. Update \eqref{eq: z-classic_admm} can be viewed as a local search scheme, where in each iteration, $R_{\rho}(z)$ is replaced by a \textit{local} approximation inside the parenthesis.} In contrast, problem \eqref{eq: admm-z-subproblem} consists of {\textit{global}} lower approximation for $R_{\rho}(z)$, which is refined over iterations. This might shed some light on why the classic ADMM cannot converge to global optimal solutions.

\subsection{Assumptions on Dual Variables} \label{sec: admm asummption}
    Our analysis builds upon the following main requirement on the sequence of dual variables.
\begin{assumption}\label{assumption: admm1}
	{Suppose $\underline{\rho}>0$ supports exact penalization for MILP \eqref{eq: MIP}, and}  $(\mu^k, \beta^k)$ are chosen such that
	\begin{enumerate}
		\item $\beta^k - \|\mu^k\|_{\infty} \geq \underline{\rho}$ for all sufficiently large $k\in \N$;
		\item $\beta^k + \|\mu^k\|_{\infty} \leq \overline{\rho}$ for all $k\in \N$ for some $\overline{\rho} > 0$.
	\end{enumerate}
\end{assumption}
\begin{remark} 
    We provide some geometric intuition regarding Assumption \ref{assumption: admm1}.
	\begin{enumerate}
	 	\item  Part 1 avoids too many loose cuts. It guarantees that for sufficiently large $k$, the peak of the AL cut reaches at least $R_{\underline{\rho}}$. Otherwise the objective of the $z$-subproblem is always a strict lower bound of $p^*$.
	 	\item Part 1 can be satisfied if $\beta^k$ is bounded away from $\|\mu^k\|_{\infty}$ by some constant. However, we do not want $\beta^k$ to go to infinity, as the resulting generalized cut is very ``slim". This is ensured by part 2.
	 	``Slim" cuts are not desirable since we will need a lot more such cuts to construct a good approximation. The constant $\overline{\rho}$ also appears in the complexity result in Theorem \ref{thm: admm_complexity}. 
	 \end{enumerate}
\end{remark}
One trivial example of Assumption \ref{assumption: admm1} is to set $\mu^k = 0$ and $\beta^k = \underline{\rho}$ for all $k$, in which case the ADMM variant reduces to AUSAL applied to the penalty problem (Theorem \ref{thm: penalized}). Another update scheme, following the classic AL-based methods, is to set 
	\begin{align}\label{eq: example_dual_update}
		\mu^{k+1} = & \Pi_{[\underline{\mu},\overline{\mu}]}\left( \mu^k + \beta^k(Ax^k+Bz^{k}) \right ),
	\end{align}
	where $\Pi_{[\underline{\mu},\overline{\mu}]}$ denotes the projection onto some hypercube $[\underline{\mu},\overline{\mu}]\subseteq \R^m$, and then let
	\begin{align}\label{eq: example_beta_update}
		\beta^{k+1} = & \min\{ \overline{\beta}, \gamma \beta^k \} \  \text{or} \  {\beta^{k+1} = \overline{\beta} \text{~for all~} k\in \N,} 
	\end{align}
	for some constants $\overline{\beta}>0$ and $\gamma \geq 1$. Clearly, Assumption \ref{assumption: admm1} can be satisfied if $\overline{\beta} \geq \underline{\rho} + \max\{ \|\underline{\mu}\|_\infty, \|\overline{\mu}\|_\infty\}$. Updates \eqref{eq: example_dual_update} and \eqref{eq: example_beta_update} generate a non-trivial AL cut, which admits an additional rotation with a potentially smaller Lipschitz constant compared to a reverse norm cut. Hence the ADMM variant is proposed in the hope that such geometric effects of AL cuts are able to shape the true value functions  $R_{\underline{\rho}}$ and $R_{\overline{\rho}}$ faster and cut off regions in $Z$ that do not contain optimal solutions. 
	
We acknowledge that in general it is hard to verify  Assumption \ref{assumption: admm1} at every iteration since $\underline{\rho}$ is unknown, and hence the ADMM variant is rather a conceptual framework. We provide some convergence properties in the next three theorems when Assumption \ref{assumption: admm1} and a slightly stronger version of it can be satisfied. The results aim to justify the usage of this conceptual framework to a reasonable level, where dual variables can be selected in a flexible way. In fact, in our numerical experiments, the projection step in \eqref{eq: example_dual_update} is skipped to encourage more diverse AL cuts.

\subsection{Convergence and Complexity}
Recall $p^*$ is the optimal value of the MILP \eqref{eq: MIP} and that $t^k =\max_{j\in [k]}\{\tilde{r}(z^k; z^{j-1}, \mu^j, \beta^j)\}$.
We first state the convergence of ADMM in the following theorem. 
\begin{theorem}\label{thm: admm convergence}
	Suppose Assumption \ref{assumption: admm1} holds. Let $\{{(x^{k+1}, z^k, t^k)}\}_{k\in \N}$ be the sequence generated by Algorithm \ref{alg: admm}, and $(x^*, z^*)$ be a limit point of $\{(x^{k+1}, z^k)\}_{k\in \N}$. The following claims hold.
	\begin{enumerate}
		\item $\{g^\top z^k + t^k\}_{k\in \N}$ converges to $p^*$ monotonically from below. 
		\item $p^* = g^\top z^* + R_{\underline{\rho}}(z^*) = g^\top z^* + R_{\overline{\rho}}(z^*)$.
		\item $(x^*, z^*)$ is an optimal solution to MILP \eqref{eq: MIP}.
	\end{enumerate}
\end{theorem}
\begin{proof}~
\begin{enumerate}
	\item We firstly prove the sequence $\{g^\top z^k + t^k\}_{k\in \N}$ converges to $p^*$ monotonically from below. Since $g^\top z^k + t^k$ is the optimal value of problem \eqref{eq: admm-z-subproblem}, whose feasible region is shrinking over $k\in \N$, we know  $\{g^\top z^k + t^k\}_{k\in \N}$ is monotone non-decreasing. 
	Since $\|\mu^k\|_\infty +\beta^k \leq \overline{\rho}$, it follows that $(\mu^k,\beta^k)\in \Lambda(\overline{\rho})$ for all for all $k\in \N$, i.e., all AL cuts are valid lower approximations of $R_{\overline{\rho}}(z)$. As a result, we have $g^\top z^k + t^k \leq \min_{z\in Z} ~g^\top z + R_{\overline{\rho}}(z) = p^*$, where the equality is due to $\overline{\rho}$ supports exact penalization. The sequence $\{g^\top z^k + t^k\}_{k\in \N}$ is non-decreasing and bounded from above by $p^*$, so it converges to some $\bar{p} \leq p^*$. Next let $\{z^{k_j}\}_{j\in \N}$ be a subsequence of $\{z^k\}_{k\in \N}$ convergent to $z^* \in Z$. By the H\"{o}lder's inequality and Assumption \ref{assumption: admm1}, we have
	\begin{align}
		& P(z^{k_{j-1}}, \mu^{k_{j-1}+1}, \beta^{k_{j-1}+1}) \notag \\
		\geq & \min_{x\in X} c^\top x + (\beta^{k_{j-1}+1} - \|\mu^{k_{j-1}+1}\|_{\infty})\|Ax+Bz^{k_{j-1}}\|_1 \notag\\
		\geq & \min_{x\in X} c^\top x + \underline{\rho} \|Ax+Bz^{k_{j-1}}\|_1 = R_{\underline{\rho}}(z^{k_{j-1}}) \label{eq: P_geq_R}
	\end{align}
	for large enough $j$.  In addition, 
	\begin{align}
		 g^\top z^{k_j} + t^{k_j} \geq & g^\top z^{k_j} + P(z^{k_{j-1}}, \mu^{k_{j-1}+1}, \beta^{k_{j-1}+1}) + \langle \mu^{k_{j-1}+1}, Bz^{k_j}-Bz^{k_{j-1}} \rangle \notag  \\
		     & - \beta^{k_{j-1}+1} \|Bz^{k_j}-Bz^{k_{j-1}}\|_1 \notag\\
		\geq & g^\top z^{k_j}+ R_{\underline{\rho}}(z^{k_{j-1}}) + \langle \mu^{k_{j-1}+1}, Bz^{k_j}-Bz^{k_{j-1}} \rangle \notag \\
		 &- \beta^{k_{j-1}+1} \|Bz^{k_j}-Bz^{k_{j-1}}\|_1.\notag 
	\end{align}
	where the first inequality is due to constraints in \eqref{eq: admm-z-subproblem} and the second inequality is due to \eqref{eq: P_geq_R}. Taking limit on both sides  gives $ \bar{p} = \lim_{j\rightarrow\infty} g^\top z^{k_j} + t^{k_j} \geq g^\top z^* + R_{\underline{\rho}}(z^*) \geq p^*,$
	where the first inequality is due to the continuity of $R_{\underline{\rho}}$, the fact that $Bz^{k_j}-Bz^{k_{j-1}}$ vanishes, and the boundedness of $\{(\mu^k,\beta^k)\}_{k\in \N}$ by Assumption \ref{assumption: admm1}. So we conclude that $g^\top z^* + R_{\underline{\rho}}(z^*)=p^*=\bar{p}$. 
	
	\item We have shown the first equality. Let $x(z^*, \underline{\rho})$ be optimal to the penalty problem $\min_{x\in X} c^\top x+ \underline{\rho}\|Ax+Bz^*\|_1$. Since $\underline{\rho}$ supports exact penalization, we have $(x(z^*, \underline{\rho}), z^*)\in =  \Argmin_{x\in X,z\in Z}\{c^\top x + g^\top z~|~Ax+Bz=0\}$. So we know $Ax(z^*, \underline{\rho})+Bz^*=0$ and $c^\top x(z^*, \underline{\rho})+g^\top z^*=p^*$. Consequently, 
	\begin{align*}
		p^* = g^\top z^* + R_{\underline{\rho}}(z^*) \leq & 	g^\top z^* + R_{\overline{\rho}}(z^*) =	g^\top z^* + \min_{x\in X} c^\top x + \overline{\rho}\|Ax+Bz^*\|_1\\
		\leq & g^\top z^* + c^\top x(z^*, \underline{\rho}) = p^*.
	\end{align*}
	This proves the second equality.
	\item Let $\{(x^{k_j+1}, z^{k_j})\}_{j\in \N}$ be the subsequence convergent to $(x^*, z^*)$. Notice that the value of $ g^\top z^{k_j} + c^\top x^{k_j+1} + \langle \mu^{k_j+1}, Ax^{k_j+1}+Bz^{k_j} \rangle + \beta^{k_j+1}\|Ax^{k_j+1}+Bz^{k_j}\|_1$ is bounded from below by $g^\top z^{k_j} + R_{\underline{\rho}} (z^{k_j})$ and from above by $ g^\top z^{k_j} + R_{\overline{\rho}} (z^{k_j})$. Assuming without loss generality that $\lim_{j\rightarrow\infty}(\mu^{k_j+1}, \beta^{k_j+1}) = (\mu^*, \beta^*)$ and taking limit on both sides of the above two inequalities give $c^\top x^* + g^\top z^* + \langle \mu^*, Ax^*+ Bz^*	\rangle + \beta^* \|Ax^*+Bz^*\|_1 = p^*,$ where the equality holds due to the second claim. It suffices to show $Ax^*+Bz^*=0$. Suppose not, then by Theorem \ref{thm: mip exact penaliztion}, $(x^*,z^*)\notin \Argmin_{x\in X, z\in Z} c^\top x+g^\top z+ \underline{\rho}\|Ax+Bz\|_1$; therefore, 
	\begin{align*}
		 p^* = & c^\top x^* + g^\top z^* + \langle \mu^*, Ax^*+ Bz^*	\rangle + \beta^* \|Ax^*+Bz^*\|_1\\
		 \geq &  c^\top x^* + g^\top z^* + \underline{\rho}\|Ax^*+Bz^*\|_1 \\
		 > & \min_{x\in X, z\in Z} c^\top x+g^\top z+ \underline{\rho}\|Ax+Bz\|_1 = p^*,
	\end{align*}
	where is a desired contradiction. This completes the proof. \qed
 \end{enumerate}
\end{proof}

In order to establish iteration complexity of Algorithm \ref{alg: admm}, we need a slightly stronger version of Assumption \ref{assumption: admm1}. 

\begin{theorem}\label{thm: admm_complexity}
     {Suppose in addition to Assumption \ref{assumption: admm1}, $\underline{\rho}-1>0$ supports exact penalization, 
    and $\beta^k - \|\mu^k\|_{\infty} \geq \underline{\rho}$ for all $k \in \N$.} Let $Z \subseteq \overline{B}_1(\bar{z}; R)$ for some $\bar{z}\in \R^d$ and radius $R>0$. Given $\epsilon> 0$, algorithm \ref{alg: admm} finds a solution $(z^K, t^K)$ of problem \eqref{eq: admm-z-subproblem} satisfying $p^* - (g^\top z^K+t^K) \leq \epsilon $ in no more than $K \leq (1+2(\underline{\rho}+\overline{\rho})\|B\|_1R\epsilon^{-1})^d$ iterations. Moreover, the following claims hold. 
	 \begin{enumerate}
	     \item Let $x(z^K, \underline{\rho}) \in \Argmin_{x\in X} c^\top x + \underline{\rho}\|Ax+Bz^{K}\|_1.$ The pair $(x(z^K, \underline{\rho}),z^K)$ is an $\epsilon$-solution to MILP \eqref{eq: MIP}. 
	     \item {
          If $\underline{\rho}$ and $\overline{\rho}$ also satisfy that 
        \begin{align} \label{assumption: admm2}
            \underline{\rho} \geq 2(\|c\|_{1} D_\infty(X) +\|g\|_{1} D_\infty(Z))\epsilon^{-1} \text{~and~} \overline{\rho} \leq \frac{3}{2}\underline{\rho} - 1,
        \end{align}
        then $(x^{K+1}, z^K)$ is an approximate solution to \eqref{eq: MIP} in the sense that
        \begin{align}
            & \|Ax^{K+1}+Bz^K\|_1 \leq \epsilon, ~\text{and}~ c^\top x^{K+1} + g^\top z^K \leq p^* +\epsilon +\|c\|_{1} D_\infty(X). \label{eq: subopt bound}
        \end{align}
        }
	 \end{enumerate}
\end{theorem}
\begin{proof}
	For all nonnegative integers $i < j$, we have
	\begin{align}\label{eq: admm_complexity3}
		& \max \{p^* - (g^\top z^j+t^j), g^\top z^j+ R_{\underline{\rho}}(z^j) - p^* \}
		\leq  R_{\underline{\rho}}(z^j)-t^j  \\
		\leq &  R_{\underline{\rho}}(z^j) - P(z^{i}, \mu^{i+1}, \beta^{i+1}) - \langle \mu^{i+1}, Bz^j-Bz^{i}\rangle + \beta^{i+1}\|Bz^j-Bz^{i}\|_1 \notag \\
		\leq &  R_{\underline{\rho}}(z^j) - R_{\underline{\rho}}(z^{i}) + (\|\mu^{i+1}\|_\infty+\beta^{i+1})\|Bz^j-Bz^{i}\|_1 \leq (\underline{\rho} + \overline{\rho})\|B\|_1\|z^j-z^{i}\|_1. \notag 
	\end{align}
	The first inequality is due to $g^\top z^j + t^j \leq p^*\leq g^\top z^j + R_{\underline{\rho}}(z^j)$, the second inequality is due to \eqref{eq: admm-z-subproblem} {(note that $j \geq i+1$)}, the third inequality is due to \eqref{eq: P_geq_R}, and the last inequality is due to $\beta^k + \|\mu^k\|_{\infty} \leq \overline{\rho}$ and $R_{\underline{\rho}}$ being $\underline{\rho}\|B\|_1$-Lipschitz. Let $K$ be the smallest index  such that $\max \{p^* - (g^\top z^K+t^K), g^\top z^K+ R_{\underline{\rho}}(z^K) - p^* \} \leq  \epsilon.$ Then we must have $\|z^i-z^j\|_1 > \epsilon/[(\underline{\rho} + \overline{\rho})\|B\|_1]$ for all $0\leq i < j\leq K-1$, since otherwise \eqref{eq: admm_complexity3} implies that
	$$\max \{p^* - (g^\top z^j+t^j), g^\top z^j+ R_{\underline{\rho}}(z^j) - p^* \}  \leq (\underline{\rho} + \overline{\rho})\|B\|_1\|z^j-z^i\|_1 \leq \epsilon,$$ contradicting to the choice of $K$. By the same argument as in the proof of Theorem \ref{thm: compexlity of lip minimization}, we can bound $K$ as claimed. Next we prove the two claims. 
    \begin{enumerate}
       \item Since $g^\top z^K+ R_{\underline{\rho}}(z^K) - p^*\leq \epsilon$, we know $(x(z^K, \underline{\rho}), z^K)$ is an $\epsilon$-optimal solution to the problem $\min_{x\in X,z\in Z} c^\top x+g^\top z+ \underline{\rho}\|Ax+Bz\|_1$. By the same proof of Theorem \ref{thm: exact_1}, we conclude that $(\tilde{x}^K,z^K)$ is an $\epsilon$-solution of MILP \eqref{eq: MIP}.  
       \item 
        Since $g^\top z^K+ R_{\underline{\rho}}(z^K) - p^*\leq \epsilon$, we know 
        \begin{align}\label{eq: admm obj bound}
            g^\top z^K + c^\top x(z^K, \underline{\rho}) + \underline{\rho}\|Ax(z^K, \underline{\rho})+Bz^K\|_1 \leq p^* + \epsilon,
        \end{align}
        Since $\underline{\rho}-1$ also supports exact penalization, 
        \begin{align*}
            p^* \leq c^\top x(z^K, \underline{\rho}) + g^\top z^{K} + (\underline{\rho}-1)\|Ax(z^K, \underline{\rho})+Bz^K\|_1.
        \end{align*}
        The above two inequalities together implies that $\|Ax(z^K, \underline{\rho})+Bz^K\|_1 \leq \epsilon$. 
        By the definition of $R_{\overline{\rho}}$, we then have 
        \begin{align}\label{eq: bound upper value function}
            R_{\overline{\rho}}(z^k) \leq c^\top x(z^K, \underline{\rho}) + \overline{\rho}\|Ax(z^K, \underline{\rho})+Bz^K\|_1\leq R_{\underline{\rho}}(z^K) + (\overline{\rho}-\underline{\rho})\epsilon. 
        \end{align}
        Denote $(\mu, \beta) = (\mu^{K+1}, \beta^{K+1})$. Invoking H\"{o}lder's inequality, it holds that  
        \begin{align}\label{eq: key ineq for finite convergence}
            & c^\top x^{K+1} + \underline{\rho} \|Ax^{K+1}+Bz^K\|_1 \\
            \leq & c^\top x^{K+1} + \langle \mu, Ax^{K+1}+Bz^K \rangle + \beta \|Ax^{K+1}+Bz^K\|_1 \notag \\
            \leq & R_{\overline{\rho}}(z^K) \leq R_{\underline{\rho}}(z^K) + (\overline{\rho}-\underline{\rho})\epsilon
            \leq p^* - g^\top z^K + (\overline{\rho}+1-\underline{\rho})\epsilon, \notag 
        \end{align}
        where the inequalities are  due to $\beta - \|\mu\|\geq \underline{\rho}$,  $\beta + \|\mu\|\leq \overline{\rho}$, \eqref{eq: bound upper value function}, and \eqref{eq: admm obj bound}, respectively. It then follows that 
        \begin{align}
            \|Ax^{K+1}+Bz^K\|_1 
            \leq & \frac{p^* - c^\top x^{K+1} - g^\top z^K}{\underline{\rho}}+ \frac{(\overline{\rho}+1-\underline{\rho})\epsilon}{\underline{\rho}} \leq \frac{\epsilon}{2} + \frac{\epsilon}{2} = \epsilon, \notag 
        \end{align}
        where the second inequality is due to\eqref{assumption: admm2}. Finally, by \eqref{eq: admm obj bound}, we see that $c^\top x^{K+1} + g^\top z^K \leq p^* + \epsilon + c^\top ( x^{K+1} - x(z^K, \underline{\rho})) \leq p^* +\epsilon + \|c\|_1 D_\infty(X)$.
    \end{enumerate} 
\end{proof}
\begin{remark}
    Some remarks follow. 
    \begin{enumerate}
        \item  Theorem \ref{thm: admm_complexity} provides an estimate on the number of ADMM iterations in order for the objective value of $z$-subproblem \eqref{eq: admm-z-subproblem} to get close to $p^*$. In particular, it is ensured that $z^K + t^K\leq p^* + \epsilon$ if we let ADMM run $K =(1+2(\underline{\rho}+\overline{\rho})\|B\|_1R\epsilon^{-1})^d$ iterations. 
        \item Theorem \ref{thm: admm_complexity} also provides two ways to recover approximate solutions to MILP \eqref{eq: MIP}. In the first case, we invoke a post-processing step by evaluating $R_{\underline{\rho}}$ (if a good estimate of $\underline{\rho}$ is available). In the second approach, we solve another $x$-subproblem \eqref{eq: admm-x-subproblem} to get $x^{K+1}$; under additional requirements on $\underline{\rho}$ and $\overline{\rho}$, $(x^{K+1}, z^k)$ is $\epsilon$-feasible with a sub-optimality bound provided in \eqref{eq: subopt bound}.
        \item We comment on the term $\|c\|_1 D_\infty(X)$ in the right-hand side of \eqref{eq: subopt bound}, which comes from the inner product $c^\top ( x^{K+1} - x(z^K, \underline{\rho}))$. Although $x^{K+1}$ and $x(z^K, \underline{\rho})$ are both closely related to $z^{K}$, we do not think it is easy to derive a better bound on their distance, especially when tools such as error bounds from convex analysis are not applicable here. Nevertheless, when $X = \prod_{i=1}^P X_i$, the diameter $D_\infty(X)$ is equal to $\max_{i\in [P]} D_{\infty} (X_i)$, which could be independent of $P$ for many applications. 
    \end{enumerate}
\end{remark}

Theorem \ref{thm: admm_complexity} does not answer the question: can ADMM itself produce an $\epsilon$-solution to MILP \eqref{eq: MIP}. We provide an affirmative answer in the next theorem. 
\begin{theorem}\label{thm: admm finite convergence}
    Let $\epsilon >0$. Suppose in addition to Assumption \ref{assumption: admm1}, $\underline{\rho}$ also satisfies
    that $\underline{\rho} \geq (\|c\|_{1} D_\infty(X) + \|g\|_{1} D_\infty(Z))\epsilon^{-1} + 1$.
    Then Algorithm \ref{alg: admm} finds an $\epsilon$-solution to MILP \eqref{eq: MIP} in a finite number of iterations. 
\end{theorem}
\begin{proof}
   By the second claim in Theorem \ref{thm: admm convergence}, there exists an index $K$ such that $g^\top z^K + R_{\overline{\rho}}(z^K)\leq p^* + \epsilon$; together with the first two inequalities in \eqref{eq: key ineq for finite convergence}, we have $ c^\top x^{K+1} + \underline{\rho} \|Ax^{K+1}+Bz^K\|_1 \leq R_{\overline{\rho}}(z^K) \leq p^* + \epsilon - g^\top z^{K}$. Now the assumed lower bound of $\underline{\rho}$ implies that $(x^{K+1}, z^K)$ is an $\epsilon$-solution to MILP \eqref{eq: MIP}. 
\end{proof}
Finally we note that the objective of the $z$-subproblem \eqref{eq: admm-z-subproblem} remains a valid lower bound of $p^*$, and if ADMM finds a feasible solution $(x^{k+1}, z^k)$, then $c^\top x^{k+1}+ g^\top z^k$ will be a valid upper bound of $p^*$. Our implementation monitors these two quantities.

\section{Numerical Experiments}\label{Section: Experiments}
We demonstrate the performance of the proposed ALM and ADMM on three classes of problems: variants of an investment planning problem \cite{schultz1998solving} (Section \ref{sec: investment}), the stochastic server location problem (SSLP) \cite{chen2022generating, ntaimo2005million} (Section \ref{sec: sslp}), and a class of structured MILPs generated with random data (Section \ref{sec: random}). We first discuss some implementation details in Section \ref{sec: implementation}. 

\subsection{Implementation Details} \label{sec: implementation}
The goal of our experiments is to demonstrate the correctness and efficiency of the proposed algorithms and assess the extent to which the scheme of adding nonconvex cuts can serve as a practical solution method. To this end, our comparisons are straightforward: we first solve a MILP instance with a given solver, then we run ALM and ADMM on the instance, where each $x/z$-subproblem is solved by the same solver. In order to evaluate the effect of nonconvex cuts alone, we do not add other well-known cuts, e.g., Benders cuts and Lagrangian cuts, though a combined implementation might improve the overall performance. 

\subsubsection{Experiment Setup}
Experiments in Sections \ref{sec: investment} and \ref{sec: random} are performed on a personal laptop with a 2.6GHz 6-Core Intel Core i7 processor and 16GB RAM, with Gurobi 9.5.2 \cite{gurobi2022gurobi} under default settings (12 threads) as the underlying MILP solver. Since large-scale SSLP instances require long solution time, experiments in Section \ref{sec: sslp} are performed on an elastic compute service (ECS) server with one Intel Xeon Platinum 8269CY CPU (104 virtual CPUs) and 768GB memory. Due to the unavailability of Gurobi on the ECS server, we use HiGHS 1.2.2 \cite{huangfu2018parallelizing} with default settings (single thread) as the underlying MILP solver. Both Gurobi and HiGHS report optimality when the relative gap is less than 0.01\%. Our codes are written in Julia 1.6.6 \cite{bezanson2017julia}, and solvers are interfaced through JuMP 1.3.0 \cite{dunning2017jump}.

\subsubsection{Implementation and Modification of Nonconvex Cuts}
Consider the ALM framework introduced in Section \ref{Section: ALM}. Fixing the current dual pair $(\lambda, \rho)$, we use AUSAL (Algorithm \ref{alg1}) to solve the primal subproblem \eqref{eq: dual function}: given $\bar{z}\in Z$, we approximate the function $R(z; \lambda, \rho)$ by adding a reverse-norm cut of the form 
\begin{align}\label{eq: oldcut}
    R(z; \lambda, \rho) \geq R(\bar{z}; \lambda, \rho) -L_{\rho}\|z-\bar{z}\|_1 
\end{align}
to the $z$-subproblem, where $R(z; \lambda, \rho) $ is defined in \eqref{eq: value function}. When the current AUSAL terminates, we update a new pair of dual variables $(\lambda^+,\rho^+)$, and start the next AUSAL. Notice that \eqref{eq: oldcut} generated with $(\lambda, \rho)$ is not valid anymore for $R(z;\lambda^+, \rho^+)$. Naively, we can remove all previous cuts when starting a new AUSAL, but this will cause a loss of historical information. Instead, we modify old cuts so that they always stay valid for the latest dual information $(\lambda^+, \rho^+)$. Assume $\rho^+ \geq \rho$. It is easy to see $ R(\bar{z}; \lambda^+,\rho^+) \geq R(\bar{z}; \lambda, \rho) - \|\lambda^+ - \lambda\|_\infty\left(\max_{x\in X} \|Ax\|_1 \right)$ and hence the following inequality is valid or all $R(z; \lambda, \rho)$ over $z\in Z$:
\begin{align}\label{eq: newcut}
	R(z; \lambda^+, \rho^+) \geq  R(\bar{z}; \lambda,\rho) - \|\lambda^+ - \lambda\|_\infty\left(\max_{x\in X} \|Ax\|_1 \right) - L_{\rho^+} \|z-\bar{z}\|_1.	
\end{align}
The new cut \eqref{eq: newcut} can be obtained by modifying the coefficient $L_{\rho}$ and constant $R(\bar{z}; \lambda,\rho)$ in \eqref{eq: oldcut}. Though these modified cuts may not be tight, empirically we observe that they help ALM to maintain a more stable lower bound. In addition, in our experiments on SSLP instances, we only maintain the latest 200 reverse norm cuts in ALM and AL cuts in ADMM.

\subsubsection{MILP Subproblems, Dual Updates, and Parameters}
In our tested problems, the $x$-subproblem in ALM and ADMM consists of a series of parallel subproblems. For ease of implementation, we solve these problems sequentially in every iteration of ALM and ADMM. Therefore, in view of solution time reported below, we expect further acceleration when more computational budgets are available.  

{
For both ALM and ADMM, the dual multiplies are initialized with zeros, and the penalty is initialized with $\rho_0 >0$. For ALM, we update dual information when the gap of AUSAL is less than 0.01\%, or \texttt{innerALM} iterations are reached. Then we set the penalty by $\rho \gets \gamma \rho$ for some $\gamma > 1$, and update multiplies by $\lambda \gets \lambda + \alpha_k (Ax^k+Bz^k)$ with $k$ being the number of AUSAL calls and  
$$\alpha_k  = \texttt{almDualStepSize}/(k\times \sqrt{2} \times \max\{1, r^k\}),$$ 
where $\texttt{almDualStepSize} > 0$ and $r^k$ is the incumbent primal residual meansured in $\ell_1$-norm. While the choice of $\alpha_k$ is motivated by Algorithm \ref{alg2}, our implementation is slightly different, as we try to terminate inefficient AUSAL calls early and maintain a more tractable sequence of penalties. For ADMM (Algorithm \ref{alg: admm}), we update the penalty $\beta \gets \gamma \beta$ in every \texttt{innerADMM} iterations, and update dual multipliers in every iteration by 
$$\mu^{k+1} \gets \mu^k + \texttt{admmDualStepSize} \times \beta \times (Ax^k+Bz^k).$$ 
Note that we do not explicitly project multipliers onto some bounded set to encourage more diverse AL cuts. In the follows, we report parameters as a 6-tuple: 
$$(\rho_0, \gamma, \texttt{innerALM}, \texttt{innerADMM}, \texttt{almDualStepSize}, \texttt{admmDualStepSize}).$$
Finally, we note that when a problem of the form \eqref{eq: MIP} is directly solved by Gurobi or HiGHS, primal and integral feasibility are determined by solvers' default tolerances. In our implementation of ALM and ADMM, we deem $(x,z) \in X \times Z$\footnote{The feasibility of $(x, z)$ with respect to $X \times Z$ is up to tolerances of the solver used.} feasible for \eqref{eq: MIP} if $\|Ax+Bz\|_1 \leq 10^{-6}$. In fact, in our tested instances, coupling constraints only involve pure integer variables and hence can always be satisfied exactly.

\subsection{Investment Planning Problems} \label{sec: investment}
We consider the following investment planning problem:
\begin{align}\label{eq: invesment}
    \min_{z}~ \{-1.5 z_1 - 4 z_2 + \mathbb{E}_{\omega} [v_\omega(z)]:~ z \in  Z \}, 
\end{align}
where $Z \subseteq \R^2$ and $v_{\omega}(z)$ is defined as 
\begin{align}
   \min_{x \in \{0,1\}^4}  \left \{-16x_1 - 19 x_2 - 23x_3-28x_4:~ \begin{bmatrix}
       2 & 3 & 4 & 5 \\  6 & 1 & 3 & 1
    \end{bmatrix} x \leq h_\omega - T z \right\} \notag 
\end{align}
with $h_\omega \in \R^2$ and $T \in \R^{2\times 2}$. The problem was firstly introduced by Schultz et al. \cite{schultz1998solving} with $Z = [0,5]^2$ and $T = I_{2\times 2}$, and its variants have been used as benchmark instances in the literature \cite{ahmed2019stochastic, gade2014decomposition,van2020converging}. We first consider the following variants tested in \cite{van2020converging}. Given an integer $S > 1$, let the two components of $h_\omega$ correspond to a lattice point over the 2-dimensional grid $\{(5+10\frac{s_1-1}{S-1},5+10\frac{s_2-1}{S-1}) :~s_1,s_2 \in [S]\}$ with equal probability. In the context of two-stage stochastic programs, the second stage contains a total of $S^2$ recourse problems. Though each subproblem in $x$ is relatively simple, in the largest tested instance where $S = 101$, there are 40,804 binary variables in the second stage, plus two general-integer variables in the first stage. We further consider two choices of the technology matrix $T$: either $T  = I_{2\times 2}$ (denoted by $\texttt{I}$), or $T = [\frac{2}{3}~\frac{1}{3}; \frac{1}{3}~\frac{2}{3}]$ (denoted by $\texttt{T}$). The results are presented in Tables \ref{table: investment_ub5_I} and \ref{table: investment_ub5_T}.

\begin{table}[h!]
\captionsetup{justification=centering}
\caption{Comparison with Gurobi for $Z = [0,5]^2 \cap \Z^2$ and $T = \texttt{I}$ \\ with
parameters (1, 1.1, 100, 50, 200, 200)} \label{table: investment_ub5_I}
\resizebox{\columnwidth}{!}{
\begin{tabular}{llllllllll}
\toprule
    & \multicolumn{3}{l}{ALM}     & \multicolumn{3}{l}{ADMM}    & \multicolumn{3}{l}{Gurobi} \\
\hline
S   & Gap    & Iter. & Time (s)   & Gap    & Iter. & Time (s)   & Gap      & Node   & Time (s)   \\
\hline
21  & 0.00\% & 37        & 9.88   & 0.00\% & 37        & 9.06   & 0.00\%   & 1      & 0.35   \\
31  & 0.00\% & 37        & 19.21  & 0.00\% & 37        & 18.47  & 0.00\%   & 1      & 0.78   \\
41  & 0.00\% & 37        & 33.81  & 0.00\% & 37        & 33.11  & 0.00\%   & 1      & 1.55   \\
51  & 0.00\% & 37        & 49.75  & 0.00\% & 37        & 50.12  & 0.00\%   & 1      & 2.77   \\
61  & 0.00\% & 37        & 72.92  & 0.00\% & 37        & 71.95  & 0.00\%   & 1      & 4.88   \\
71  & 0.00\% & 37        & 95.15  & 0.00\% & 37        & 98.80  & 0.00\%   & 1      & 8.35   \\
81  & 0.00\% & 37        & 123.69 & 0.00\% & 37        & 127.82 & 0.00\%   & 1      & 15.98  \\
91  & 0.00\% & 37        & 157.53 & 0.00\% & 37        & 161.15 & 0.00\%   & 1      & 22.04  \\
101 & 0.00\% & 37        & 195.81 & 0.00\% & 37        & 201.96 & 0.00\%   & 1      & 31.13  \\
\hline
Avg. & 0.00\% &37        & 84.19  & 0.00\% & 37        & 85.83 & 0.00\%   &   1     & 9.76  \\
\bottomrule
\end{tabular}
}
\end{table}

\begin{table}[h]
\captionsetup{justification=centering}
\caption{Comparison with Gurobi for $Z = [0,5]^2 \cap \Z^2$ and $T = \texttt{T}$ \\ with
parameters (1, 1.1, 100, 50, 200, 200)} \label{table: investment_ub5_T}
\resizebox{\columnwidth}{!}{
\begin{tabular}{llllllllll}
\toprule
     & \multicolumn{3}{l}{ALM}     & \multicolumn{3}{l}{ADMM}    & \multicolumn{3}{l}{Gurobi} \\
\hline
S    & Gap    & Iter.  & Time (s)   & Gap    & Iter.  & Time (s)   & Gap      & Node  & Time (s)   \\
\hline
21   & 0.00\% & 37        & 10.16  & 0.00\% & 37        & 9.26   & 0.00\%   & 2987  & 68.91   \\
31   & 0.00\% & 37        & 21.76  & 0.00\% & 37        & 20.96  & 0.01\%   & 2874  & 53.59   \\
41   & 0.00\% & 37        & 32.84  & 0.00\% & 37        & 32.94  & 0.00\%   & 3859  & 129.19  \\
51   & 0.00\% & 37        & 49.64  & 0.00\% & 37        & 50.10  & 0.01\%   & 3617  & 289.76  \\
61   & 0.00\% & 37        & 74.19  & 0.00\% & 37        & 74.81  & 0.01\%   & 3825  & 150.14  \\
71   & 0.00\% & 37        & 95.11  & 0.00\% & 37        & 99.47  & 0.00\%   & 2875  & 268.72  \\
81   & 0.00\% & 37        & 123.02 & 0.00\% & 37        & 126.08 & 0.01\%   & 3746  & 228.46  \\
91   & 0.00\% & 37        & 162.84 & 0.00\% & 37        & 173.63 & 0.01\%   & 2578  & 214.24  \\
101  & 0.00\% & 37        & 190.75 & 0.00\% & 37        & 208.48 & 0.01\%   & 2444  & 351.39  \\
\hline
Avg. & 0.00\% &  37         & 84.48  & 0.00\% & 37          & 88.41  & 0.01\%   & 3201      & 194.93  \\
\bottomrule
\end{tabular}
}
\end{table}

For $T = \texttt{I}$, Gurobi is able to find the optimal solution at the root node, and hence the number of explored node is 1 across all instances. We observe that Gurobi invokes many heuristics at the root node, while another solver, HiGHS, does not report optimality at the root node for any instance in Table \ref{table: investment_ub5_I}. The proposed ALM and ADMM are slower, partly because a large number of subproblems in $x$ are solved sequentially in our implementation. Nevertheless, for all instances, both ALM and ADMM close the gap with the same optimal objective value as Gurobi. For $T = \texttt{T}$, the problems become harder. Gurobi explored a few thousands of nodes before reporting optimal solutions. In contrast, ALM and ADMM are able to locate and verify optimality around 55\% faster than Gurobi on average, even though recourse subproblems are solved sequentially. 

\begin{table}[h]
\captionsetup{justification=centering}
\caption{Comparison with Gurobi for $Z = [0,10]^2 \cap \Z^2$ and $T = \texttt{T}$ \\ with parameters (1, 1.1, 100, 100, 200, 0.01)} \label{table: investment_ub10_T}
\resizebox{\columnwidth}{!}{
\begin{tabular}{llllllllll}
\toprule
     & \multicolumn{3}{l}{ALM}     & \multicolumn{3}{l}{ADMM}    & \multicolumn{3}{l}{Gurobi} \\
\hline
S    & Gap    & Iter. & Time (s)  & Gap    & Iter. & Time (s)   & Gap     & Node   & Time (s)   \\
\hline
21   & 0.00\% & 107       & 38.21  & 0.00\% & 99        & 42.07  & 0.00\%  & 2103   & 14.71   \\
31   & 0.00\% & 107       & 74.03  & 0.00\% & 105       & 88.35  & 0.62\%  & 12353  & 1800.01 \\
41   & 0.00\% & 107       & 113.77 & 0.00\% & 101       & 137.30 & 0.05\%  & 43435  & 1800.01 \\
51   & 0.00\% & 107       & 161.41 & 0.00\% & 102       & 209.29 & 0.16\%  & 30359  & 1800.02 \\
61   & 0.00\% & 107       & 234.30 & 0.00\% & 105       & 308.62 & 0.51\%  & 10681  & 1800.01 \\
71   & 0.00\% & 107       & 298.71 & 0.00\% & 102       & 396.42 & 0.02\%  & 37002  & 1800.03 \\
81   & 0.00\% & 107       & 382.68 & 0.00\% & 102       & 514.50 & 0.46\%  & 11571  & 1800.01 \\
91   & 0.00\% & 107       & 497.18 & 0.00\% & 105       & 692.33 & 0.69\%  & 11114  & 1800.02 \\
101  & 0.00\% & 107       & 585.47 & 0.00\% & 102       & 826.80 & 0.03\%  & 39720  & 1800.05 \\
\hline 
Avg. & 0.00\% & 107          & 265.08 & 0.00\% & 103           & 357.30 & 0.28\%  & 22038      & 1601.65 \\
\bottomrule
\end{tabular}
}
\end{table}

Another interesting observation is that, ALM and ADMM always terminate at the 37th iteration. Note that the feasible region $Z = [0,5]^2 \cap \Z^2$ consists of 36 points in $\R^2$. The proposed algorithms enumerate the feasible region in the first 36 iterations, and use one more iteration to verify optimality. Though this worst-case complexity is not very promising, we further enlarge the feasible region of $z$ to  $Z = [0,10]^2 \cap \Z^2$ so that there are 121 feasible solutions. We report the same metrics in Table \ref{table: investment_ub10_T}. The proposed ALM and ADMM terminate successfully in less than 121 iterations for all instances and take 265.08 and 357.30 seconds on average, respectively. In contrast, without structural knowledge on the problem data, Gurobi only solves the smallest instance within 1800 seconds.  

\subsection{Stochastic Server Location Problems }  \label{sec: sslp}
The stochastic server location problem (SSLP) \cite{chen2022generating, ntaimo2005million} is a classic two-stage MILP that can be cast as follows: 
\begin{align*}\label{eq: sslp}
    \min_{z, \{x^p,s^p\}_{p\in [P]}}   \quad & \sum_{j=1}^m c_j z_j + \sum_{p\in [P]} \mathrm{prob}_p\left( \sum_{j=1}^m q_{0j} s^p_j - \sum_{i=1}^n \sum_{j=1}^m q_{ij} x^p_{ij} \right) \\
    \mathrm{s.t.} \quad &  \sum_{i=1}^n d_{ij} x^{p}_{ij} \leq u z_j + s^p_{j}, \quad  j\in [m], ~p \in [P], \\
    & \sum_{j=1}^m x^{p}_{ij} = h^p_i, \quad \quad \quad \quad \quad  i\in [n], ~p \in [P], \\
    & z_j, x^p_{ij} \in \{0,1\},  s^p_j \geq 0,   \quad  i \in [n],~j \in [m], ~p \in [P].
\end{align*}
The problem aims to allocate servers to $m$ potential sites and meet the demand of $n$ potential clients. 
In the first stage, a decision maker needs to allocate servers at $m$ potential sites ($z_j$ = 1 if and only if a server is located at site $j$) associated with a cost $c_j$ for $j \in [m]$. Then in each scenario $p$ in the second stage, the availability of client $i$ is observed and expressed by a vector $h^p$ with $h^p_i = 1$ if and only if client $i$ is present in scenario $p$. The variable $x^p_{ij}=1$ if and only if client $i$ is served at site $j$ in scenario $p$. Each allocated server has $u$ units of resource, and client $i$ uses $d_{ij}$ units of resource at site $j$ and generate revenue $q_{ij}$;  shortage of resource at site $j$ is modeled by a continuous variable $s^p_j\geq 0$ and penalized by a cost $q_{0j}$. 

We generate data the same way as in \cite{chen2022generating}. Each $c_j$ is uniformly sampled from $\{40, 41, \cdots, 80\}$, each $d_{ij}$ is uniformly sampled from $\{0, 1, \cdots, 25\}$, and each $h^p_i$ is 0 or 1 with equal probability; we then set $\mathrm{prob}_p = 1/P$, $q_{0j} = 1000$, and $u = \frac{1}{m} \sum_{j=1}^m \sum_{i=1}^n d_{ij}$. We fix the number of second-stage scenarios to $P = 50$, and experiment with two sets of $(m,n)$ pairs. In the first set, the number of potential servers is relatively small: $\texttt{mSmall}= \{(5, 100), (10, 100), (15, 100)\}$; in the second set, we test on problem scales considered in  \cite{chen2022generating}: $\texttt{mLarge}= \{(20, 100), (30, 70), (40, 50), (50, 40)\}$. We limit ALM and ADMM iterations by 2000, and set a time limit of 7200 seconds for all solution methods. 

The results are presented in Tables \ref{table: sslp_small} and \ref{table: sslp_large}. On \texttt{mSmall} instances, the average gap of 9 instances is 1.81\% for ALM and 1.65\% for ADMM, while HiGHS has a slightly higher gap of 2.2\%. HiGHS also solves two more instances to optimality within the time limit. The \texttt{mLarge} instances are more challenging for the proposed methods. ALM and ADMM obtain an average of gap of 22.09\% and 14.41\%, respectively, and HiGHS achieves a better average gap of 6.45\%. We also observe that ADMM tends to perform better than ALM. Such an advantage might due to 1) historical cuts are effectively preserved in ADMM, and 2) linear terms in AL cuts can help the method escape non-optimal regions faster. 

We note that the state-of-the-art approach by Chen and Luedtke \cite{chen2022generating}  reports optimality on \texttt{mLarge} instances with an average solution time of 1469 seconds for $P = 50$ and 4348 seconds for $P = 200$. While not immediately comparable to theirs, our results suggest that the proposed methods can be preferable over a black-box MILP solver, at least when the dimension or feasible region of $Z$ is not too large. Directions for further acceleration include parallel implementation of $x$-subproblems, instance-specific parameter tuning (note that we fix a set of parameters in each table, while dual updates can affect the performance of AL-based methods drastically), and combination with existing linear cuts. 

\begin{table}[h]
\captionsetup{justification=centering}
\caption{Comparison with HiGHS on \texttt{mSmall} instances \\ with parameters (1, 1.25, 50, 50, 50, 50)} \label{table: sslp_small}
\resizebox{\columnwidth}{!}{
\begin{tabular}{cclllllllll}
\toprule 
          &          & \multicolumn{3}{l}{ALM}  & \multicolumn{3}{l}{ADMM} & \multicolumn{3}{l}{HiGHS} \\
\hline 
$m$-$n$-$P$     & Instance & Gap    & Iter. & Time (s)    & Gap    & Iter. & Time (s)    & Gap     & Node  & Time (s)   \\
\hline 
          & 1        & 0.00\% & 86    & 225.73  & 0.00\% & 17    & 25.65   & 0.00\%  & 176   & 78.56   \\
5-100-50  & 2        & 0.00\% & 81    & 221.86  & 0.00\% & 42    & 81.25   & 0.00\%  & 955   & 213.43  \\
          & 3        & 0.00\% & 135   & 570.40  & 0.00\% & 54    & 82.59   & 0.00\%  & 735   & 111.68  \\
\hline 
          & 1        & 1.53\% & 2000  & 4381.55 & 1.85\% & 2000  & 3531.80 & 0.00\%  & 3266  & 1958.28 \\
10-100-50 & 2        & 3.15\% & 2000  & 4728.14 & 2.37\% & 2000  & 3900.31 & 6.31\%  & 17918 & 7200.39 \\
          & 3        & 2.18\% & 2000  & 4822.01 & 0.55\% & 2000  & 3769.56 & 0.01\%  & 3513  & 1782.54 \\
\hline 
          & 1        & 3.90\% & 1529  & 7200.00 & 3.68\% & 2000  & 7002.99 & 5.35\%  & 9178  & 7200.00 \\
15-100-50 & 2        & 2.28\% & 1490  & 7200.00 & 1.49\% & 1387  & 7200.00 & 4.81\%  & 9781  & 7200.08 \\
          & 3        & 3.22\% & 1790  & 7200.00 & 4.89\% & 1776  & 7200.00 & 3.50\%  & 14453 & 7200.01 \\
\hline
Avg.      &          & 1.81\% & 1235     & 4061.08 & 1.65\% &   1253    & 3643.79 & 2.22\%  &  6664     & 3660.55 \\
\bottomrule
\end{tabular}
}
\end{table}

\begin{table}[h]
\captionsetup{justification=centering}
\caption{Comparison with HiGHS on \texttt{mLarge} instances \\ with parameters (1, 1.25, 50, 50, 20, 20)} \label{table: sslp_large}
\resizebox{\columnwidth}{!}{
\begin{tabular}{cclllllllll}
\toprule
                             &          & \multicolumn{3}{l}{ALM}       & \multicolumn{3}{l}{ADMM}      & \multicolumn{3}{l}{HiGHS} \\
\hline
$m$-$n$-$P$                        & Instance & Gap     & Iter. & Time (s)    & Gap     & Iter. & Time (s)     & Gap      & Node & Time (s)     \\
\hline
                             & 1        & 7.41\%  & 269       & 7200.00 & 5.48\%  & 516       & 7200.00 & 8.94\%   & 6680 & 7200.05 \\
20-100-50                    & 2        & 8.59\%  & 640       & 7200.00 & 7.37\%  & 1244      & 7200.00 & 7.29\%   & 5276 & 7200.07 \\
                             & 3        & 9.37\%  & 351       & 7200.00 & 7.64\%  & 895       & 7200.00 & 6.89\%   & 5460 & 7200.04 \\
\hline
                             & 1        & 25.45\% & 290       & 7200.00 & 15.89\% & 738       & 7200.00 & 8.32\%   & 3546 & 7200.03 \\
30-70-50                     & 2        & 20.69\% & 209       & 7200.00 & 17.99\% & 558       & 7200.00 & 13.08\%  & 2892 & 7200.01 \\
                             & 3        & 24.75\% & 229       & 7200.00 & 8.57\%  & 189       & 7200.00 & 3.44\%   & 2604 & 7200.02 \\
\hline
                             & 1        & 21.87\% & 197       & 7200.00 & 24.26\% & 382       & 7200.00 & 5.66\%   & 1634 & 7200.04 \\
40-50-50                     & 2        & 36.30\% & 187       & 7200.00 & 22.17\% & 359       & 7200.00 & 9.68\%   & 1942 & 7200.20 \\
                             & 3        & 18.79\% & 196       & 7200.00 & 9.25\%  & 332       & 7200.00 & 2.71\%   & 1779 & 7200.21 \\
\hline
                             & 1        & 44.22\% & 170       & 7200.00 & 27.23\% & 286       & 7200.00 & 8.14\%   & 1192 & 7200.10 \\
50-40-50 & 2        & 30.11\% & 71        & 7200.00 & 7.98\%  & 306       & 7200.00 & 0.00\%   & 117  & 2538.02 \\
\multicolumn{1}{l}{}         & 3        & 17.59\% & 167       & 7200.00 & 19.07\% & 314       & 7200.00 & 3.23\%   & 1435 & 7200.01 \\
\hline
\multicolumn{1}{l}{Avg.}     &          & 22.09\% & 248       & 7200.00 & 14.41\% & 510       & 7200.00 & 6.45\%   & 2880 & 6811.57 \\
\bottomrule
\end{tabular}
}
\end{table}

{
\subsection{A Class of Random MILPs} \label{sec: random}
Our experiments in the previous two subsections suggest that the scheme of adding nonconvex cuts is indeed practical, and can be even preferable, when the feasible region of $Z$ is not too large and subproblems in $x$ are relatively easy. To further demonstrate how the proposed methods can take advantages of such structures, we consider a class of MILPs in the form of \eqref{eq: MIP} generated as follows. Given an integer $P > 0$, for each $p\in [P]$, we create a polytope $X_p = \{ x \in [0,2]^{50}:~ E_p x = f_p , x_p \in \Z^{30} \times \R^{20}\}$ and an objective vector $c_p \in \R^{50}$. The matrix $E_p \in \R^{30 \times 50}$ and the vector $c_p$ have standard Gaussian entries, and $f_p = E_p \bar{x}_p \in \R^{30}$,  where $\bar{x}_i$ is uniformly sampled from $\{0,1,2\}$ for $i \in \{1,\cdots, 30\}$ and $[0,2]$ for $i \in \{31, \cdots, 50\}$. Then from each block $p$, we introduce a copy of the first three components of $x_p$, denoted by $(z_{p1}, z_{p2}, z_{p3})$.  We then denote all copied variables by $z$, and let $Z = \{z \in [0,2]^{3P} \cap \Z^{3P}:~Gz = h\}$, where $G \in \R^{(3P-50) \times 3P}$ has standard Gaussian entries, and $h = G\bar{z}$, where each component of $\bar{z}$ has the same value as its copy in $\bar{x}_p$ so that the problem is feasible. By construction, the coupling constraints $Ax+Bz = 0$ has the form $-[x_{p}]_i + z_{pi} = 0$ for $i\in [3]$, $p\in [P]$, where $[x_{p}]_i$ denotes the $i$-th component of $x_p$. The objective vector of $z$ is set to zeros. Written in its extensive form \eqref{eq: MIP}, the problem has $20P$ continuous variables, $33P$ integer variables, and $36P - 50$ equality constraints. 
}

For each $P \in \{50, 100, 200, 500\}$, we generate 3 instances and report results in Table \ref{table: random}. The proposed ALM and ADMM successfully terminate for all 12 instances with zero duality gaps, while Gurobi fails to find feasible solutions when $P\in \{200, 500\}$ in 1800 seconds. We do not report iterations of ALM and ADMM because both of them terminate in exactly 2 iterations for all generated instances: the first iteration finds the optimal solution, and the second iteration verifies the solution is indeed globally optimal by adding a nonconvex cut at the same point. 

Note that the matrix $G$ is nearly square as $P$ increases. Our construction deliberately reduces the number of feasible solutions in $Z$ so that once a feasible (probably optimal as well) $z$ is found, the rest problem in $x$ can be easily solved. We note that other linear cuts, i.e., Benders and Lagrangian cuts, should also terminate the algorithm in two iterations: since the optimal solution is found in the first iteration, which cannot be cut off by valid cuts, the second iteration should close the gap. On the other hand, Gurobi as a general-purpose solver does not pass such structural information to branch-and-bound. The results have no intention to claim the superiority of ALM and ADMM over Gurobi, or any MILP solvers, but rather demonstrate that the proposed methods can significantly benefit from problem structures.  

\begin{table}[h]
\captionsetup{justification=centering}
\caption{Comparison with Gurobi on a class of MILPs with random data \\ with parameters (1, 1.1, 100, 50, 200, 200)} \label{table: random}
\resizebox{\columnwidth}{!}{
\begin{tabular}{cclllllll}
\toprule
                     &          & \multicolumn{2}{l}{ALM} & \multicolumn{2}{l}{ADMM} & \multicolumn{3}{l}{Gurobi}   \\
\hline
P                    & Instance  & Gap        & Time (s)      & Gap        & Time (s)       & Gap    & Node  & Time (s)    \\
\hline 
 & 1        & 0.00\%     & 8.88       & 0.00\%     & 8.69        & 0.00\% & 38369 & 146.88  \\
50                   & 2        & 0.00\%     & 8.97       & 0.00\%     & 8.71        & 0.00\% & 50905 & 163.85  \\
                     & 3        & 0.00\%     & 7.98       & 0.00\%     & 7.51        & 0.00\% & 37550 & 264.91  \\
\hline 
                     & 1        & 0.00\%     & 18.78      & 0.00\%     & 19.15       & 0.00\% & 31679 & 674.83  \\
100                  & 2        & 0.00\%     & 20.90      & 0.00\%     & 21.38       & 0.00\% & 40892 & 863.40  \\
                     & 3        & 0.00\%     & 17.78      & 0.00\%     & 17.78       & 0.00\% & 28346 & 507.21  \\
\hline 
                     & 1        & 0.00\%     & 55.71      & 0.00\%     & 53.47       & -      & 20005 & 1800.27 \\
200                  & 2        & 0.00\%     & 52.24      & 0.00\%     & 51.68       & -      & 30363 & 1800.17 \\
                     & 3        & 0.00\%     & 53.93      & 0.00\%     & 53.84       & -      & 25722 & 1800.56 \\
\hline 
                     & 1        & 0.00\%     & 1108.85    & 0.00\%     & 1108.05     & -      & 1317  & 1801.52 \\
500                  & 2        & 0.00\%     & 326.92     & 0.00\%     & 328.01      & -      & 1447  & 1801.45 \\
                     & 3        & 0.00\%     & 448.33     & 0.00\%     & 442.67      & -      & 3945  & 1800.36 \\
\hline
Avg. &          & 0.00\%     & 177.44     & 0.00\%     & 176.75      &   -     &  25878    & 1118.78 \\
\bottomrule
\end{tabular}
}
\end{table}

\section{Concluding Remarks} \label{Section: Conclusion}
In this paper, we study generic MILP problems with two blocks of variables $x$ and $z$. We propose an algorithm named AUSAL that alternatively updates $x$ and $z$ in the augmented Lagrangian function, which can be further directly embedded into the penalty method or ALM. We also propose a single-looped ADMM variant, which is built upon the AL cut introduced in \cite{ahmed2019stochastic,zhang2019stochastic}. Different from the procedure used in the previous two references, we obtain an AL cut by solving a single augmented Lagrangian relaxation in variable $x$; compared to existing ADMM works, our ADMM variant allows a more flexible update scheme for the dual variable and penalty, and is guaranteed to converge to a global optimal solution with iteration complexity estimates. When certain block-angular structure is present, the update of $x$ can be further decomposed and solved in parallel in both algorithms. 

We conduct numerical experiments on three classes of MILPs and demonstrate that the proposed methods exhibit advantages on structured problems over the state-of-the-art MILP solvers. Admittedly, the update of $z$ variable in both algorithms requires solving a MILP problem with an increasing size of variables and constraints, which can be the computational bottleneck for large and dense problems. We are interested in investigating more practical subproblem oracles, i.e., managing a controllable size of cuts, or new methodologies to approximate the dependency between $x$ and $z$. We leave these in the future work.

\appendix
\section{Missing Proofs} \label{appendix: proof}
\subsection{Proof of Theorem \ref{thm: compexlity of lip minimization}} \label{sec: proof of lip minimization complexity}
We first present a useful lemma. 
\begin{lemma}\label{lemma: local_opt}
	Let $z^k$ be an iterate generated by Algorithm \ref{alg: lip minimization}. Then $Q(z) - \underline{R}(z;Z_{k}) \leq \epsilon$ for all $z\in Z$ such that $\|z-z^k\|_1\leq \epsilon/(2L)$.
\end{lemma}
\begin{proof}
	For all $z \in Z$ and $\|z-z^k\|_1\leq \epsilon/(2L)$, we have $Q(z) - \underline{R}(z; Z_{k}) \leq  Q(z) - Q(z^k) + L \|z-z^k\|_1 \leq  2L \|z-z^k\|_1 \leq \epsilon,$
	where the first inequality is due to $ \underline{R}(z; Z_{k})\geq Q(z^k)-L \|z-z^k\|_1$, and the second inequality is due to $Q$ being $L$-Lipschitz. 
\end{proof}
\begin{proof}[Proof of Theorem \ref{thm: compexlity of lip minimization}]
    Let $\{z^{k_j}\}_{j\in \N}$ be a subsequence convergent to some optimal solution $z^*$ of \eqref{eq: lip minimization} {if we do not terminate Algorithm \ref{alg: lip minimization}}. Since $\lim_{j\rightarrow \infty}z^{k_j}=z^*$, we have $\|z^{k_{j+1}} - z^{k_j}\|_1 \leq \epsilon/(2L)$ for all large enough $j\in \N$. Notice $k_{j+1}\geq k_j+1$; by Lemma \ref{lemma: local_opt}, it follows $Q(z^{k_{j+1}}) - t^{k_{j+1}} \leq Q(z^{k_{j+1}}) - \underline{R}(z^{k_{j+1}}; Z_{k_{j}})\leq \epsilon$. Now let $T$ be the first index such that $Q(z^T)-t^T \leq \epsilon$, which is sufficient to ensure the termination of Algorithm \ref{alg: lip minimization}. For all $0\leq i < j \leq T-1$, we claim that $\|z^i-z^j\|_1> \epsilon/(2L)$: suppose not, then Lemma \ref{lemma: local_opt} suggests that $Q(z^{j}) -t^j  \leq Q(z^{j}) - \underline{R}(z^j; Z_{i}) \leq \epsilon$, contradicting to the choice of $T$. Let $r = \epsilon/(4L)$. Since $z^i \in Z$ for all $0\leq i\leq T-1$ and $Z\subseteq \overline{B}_1(\bar{z}; R)$, $\bigcup_{i=0}^{T-1} \overline{B}_1(z^i;r) \subseteq Z + \overline{B}_1(0;r) \subseteq \overline{B}_1(\bar{z}; R+r);$ since the balls $\{\overline{B}_1(z^i;r)\}_{0\leq i\leq T-1}$ are disjoint, it follows
    \begin{equation*}
         \mathrm{Vol}\left(\bigcup_{i=0}^{T-1} \overline{B}_1(z^i;r)\right) = T \mathrm{Vol}\left((\overline{B}_1(0;r)\right)\leq \mathrm{Vol}\left(\overline{B}_1(\bar{z}; R+r)\right),
    \end{equation*}
    where $\mathrm{Vol}(\cdot)$ returns the volume of the argument, and thus 
    \begin{equation*}
        T \leq \frac{\mathrm{Vol}(\overline{B}_1(\bar{z};R+r))}{\mathrm{Vol}(\overline{B}_1(0;r))} = \left(\frac{R+r}{r}\right)^d = ( 1 + 4LR\epsilon^{-1})^d.
    \end{equation*}
    This completes the proof. 
\end{proof}

\subsection{Proof of Theorem \ref{thm: dual complexity}}\label{sec: proof of dual compex}
We first state a standard lemma regarding the progress in objective value.
\begin{lemma}\label{lemma: dual complexity}
	Let $\{(\lambda^k,\rho^k)\}_{k\in \N}$ be the sequence of iterates generated by Algorithm \ref{alg2}. Then for all $K\geq 1$, it holds that
	\begin{align*}
		\min_{k\in [K]} p^* - d(\lambda^k,\rho^k )\leq  \frac{M}{\sqrt{2}} \frac{d_0^2 + \sum_{k=1}^K \tau_k^2}{\sum_{k=1}^K \tau_k} + \epsilon_p.	\end{align*}
	\end{lemma}
\begin{proof}
	To simplify notation, we denote $w^k = (\lambda^k,\rho^k)$, and denote $w^*=(\lambda^*,\rho^*)$ to be the maximizer in the definition of $d_0$. We also denote the $\epsilon$-subgradient of $(-d)(w^k)$ by $g^k$, and it holds
	\begin{equation}\label{eq: dual subgradient bd}
	    \|g^k\|^2_2 = \|Ax^k+Bz^k\|_2^2 + \|Ax^k+Bz^k\|_1^2 \leq 2\|Ax^k+Bz^k\|^2_1\leq 2M^2.
	\end{equation}
	Use the fact that $g^k\in \partial_\epsilon (-d)(w^k)$, we have $\|w^{k+1}-w^*\|^2_2 =  \|w^{k} - \alpha_k g^k-w^*\|^2_2 \leq \|w^k - w^*\|^2_2 + \alpha_k^2 \|g_k\|^2_2 - 2\alpha_k (p^* - d(w^k)-\epsilon_p)$. Summing over $k=1,\cdots, K$, we see $2 \left(\min_{k\in [K]} p^* - d(\lambda^k,\rho^k) -\epsilon_p\right) \sum_{k=1}^K \alpha_k$ is bounded from above by
	\begin{align}
        & 2\sum_{k=1}^K \alpha_k (p^* - d(\lambda^k,\rho^k) -\epsilon_d) \leq  \|w^1-w^*\|_2^2 + \sum_{k=1}^K \alpha_k^2 \|g^k\|_2^2 \notag \\
        \leq & \|w^1-w^*\|_1^2 + \sum_{k=1}^K \alpha_k^2 2\|Ax^k+Bz^k\|_1^2 \notag =  d_0^2 + \sum_{k=1}^K \tau_k^2,	
    \end{align}
	which further implies 
	\begin{align*}
		\min_{k\in [K]} p^* - d(\lambda^k,\rho^k )\leq \frac{d_0^2 + \sum_{k=1}^K \tau_k^2}{2\sum_{k=1}^K \alpha_k} + \epsilon_p \leq \frac{M}{\sqrt{2}} \frac{d_0^2 + \sum_{k=1}^K \tau_k^2}{\sum_{k=1}^K \tau_k} + \epsilon_p.
	\end{align*} 
\end{proof}

\begin{proof}[Proof of Theorem \ref{thm: dual complexity}]
	The first claim follows from Lemma \ref{lemma: dual complexity} and the choices of $\tau_i$ and $K$ so that:
	\begin{align*}
		\frac{M}{\sqrt{2}} \frac{d_0^2 + \sum_{k=1}^K \tau_k^2}{\sum_{k=1}^K \tau_k} =& \frac{M}{\sqrt{2}} \frac{2d_0^2}{\sqrt{K}d_0} = \frac{\sqrt{2} Md_0}{\sqrt{K}}\leq \epsilon_d.
	\end{align*}
	Let $\gamma_k =\alpha_k /\tau_k$. Recall the bound in \eqref{eq: dual subgradient bd}, and  we have $\gamma_k \|g^k\|_2\leq 1$;
	the second claim is then proved in {\cite[Theorem 7.4]{ruszczynski2011nonlinear}}. 
\end{proof}

\subsection{Proof of Theorem \ref{thm: finite terminate}}\label{sec: proof of finite term}
\begin{proof}[Proof of Theorem \ref{thm: finite terminate}]
   	Firstly notice that according to the penalty update, we have $\rho^k = \rho^1 + \sum_{j=2}^k \rho^{j}-\rho^{j-1} = \rho^1 + \sum_{j=2}^k \max\{\|\lambda^j\|_{\infty}, \alpha_j \|Ax^j+Bz^j\|_1 \} \geq \rho^1 +  \sum_{j=2}^k  \alpha_j \|Ax^j+Bz^j\|_1=   \rho^1 + (k-1)\tau.$
    For the purpose of contradiction, suppose $\|Ax^k+Bz^k\|_1 > \epsilon_p$ for all $k\in \N$, and thus Algorithm \ref{alg3} will generate an unbounded sequence $\{\rho^k\}_{k\in \N}$.
	Let $(\lambda^*,\rho^*)$ be an optimal solution to the dual problem \eqref{eq: dual problem}. Then we have $ \|\lambda^{k+1}-\lambda^*\|_2^2$ is bounded from above by
	\begin{align}
		  & \|\lambda^k- \lambda^*\|_2^2 + \alpha_k^2 \|Ax^k+Bz^k\|_2^2 + 2\alpha_k\left(d(\lambda^k, \rho^k) - p^* + \epsilon_p+ \|Ax^k+Bz^k\|_1 (\rho^*-\rho^k)\right)\notag  \\
		  & \leq \|\lambda^k- \lambda^*\|_2^2 + \alpha_k^2 \|Ax^k+Bz^k\|_2^2 + 2\alpha_k\left( \epsilon_p+ \|Ax^k+Bz^k\|_1 (\rho^*-\rho^k)\right),\notag 
	\end{align}
	where the inequality is due to \eqref{eq: dual_grad} and $d(\lambda^k,\rho^k)\leq p^*$. By the definition of $\alpha_k$ and the fact that $\|Ax^k+Bz^k\|_1 > \epsilon_p$, we further have
	\begin{align}
		\|\lambda^{k+1}-\lambda^*\|_2^2 
		\leq   \|\lambda^k-\lambda^*\|_2^2 + \tau^2 + 2\tau + 2\tau\rho^* - 2\tau \rho^k. \label{eq: dual_bound_3}
	\end{align}
	Notice that when $\rho^k \geq \rho^* + \tau/2 +1$, we have $\|\lambda^{k+1}-\lambda^*\|_2 \leq  \|\lambda^k-\lambda^*\|_2$, and thus the dual sequence $\lambda^k$ stays bounded; now letting $k\rightarrow \infty$ on \eqref{eq: dual_bound_3}, the left-hand side is nonnegative while the right-hand side goes to $-\infty$, which is a desired contradiction. 
\end{proof}

\section*{Acknowledgments}
We would like to thank the authors of \cite{camisa2018primal} for making their primal-decomposition code available and the authors of \cite{zhang2019stochastic} for the discussion on AL cuts. A part of this work was done during an internship of Alibaba (US) Innovation Research.
\bibliographystyle{siamplain}
\bibliography{ref.bib}
\end{document}